\documentclass[aop,preprint]{imsart}

\RequirePackage{amsthm,amsmath}
\RequirePackage[numbers]{natbib}
\RequirePackage[colorlinks,citecolor=blue,urlcolor=blue]{hyperref}

\usepackage{latexsym}

\usepackage{amsmath,amssymb,dsfont,amsfonts}
\usepackage{mathrsfs}
\usepackage{verbatim}
\usepackage{graphicx}
\usepackage{graphics}
\usepackage{booktabs}
\usepackage{enumitem}

\newtheorem{proposition}{Proposition}[section]
\newtheorem{theorem}{Theorem}[section]
\newtheorem{definition}{Definition}[section]
\newtheorem{corollary}{Corollary}[section]
\newtheorem{lemma}{Lemma}[section]
\newtheorem{remark}{Remark}[section]
\newtheorem{example}{Example}[section]

\DeclareMathOperator*{\esssup}{ess\,sup} 

\numberwithin{equation}{section}

\arxiv{arXiv:0000.0000}

\startlocaldefs
\numberwithin{equation}{section}
\theoremstyle{plain}

\endlocaldefs

\begin{document}

\begin{frontmatter}
\title{A weak version of path-dependent functional It\^o calculus}
\runtitle{A weak version of path-dependent functional It\^o calculus}

\begin{aug}
\author{\fnms{Dorival} \snm{Le\~ao}\thanksref{t1}\ead[label=e1]{leao@estatcamp.com.br}},
\author{\fnms{Alberto} \snm{Ohashi}\thanksref{m2}\ead[label=e2]{amfohashi@gmail.com}}
\and
\author{\fnms{Alexandre B.} \snm{Simas}\thanksref{m2}
\ead[label=e3]{alexandre@mat.ufpb.br}
\ead[label=u1,url]{http://www.foo.com}}

\runauthor{D. Le\~ao, A. Ohashi and A.B Simas}

\affiliation{Universidade de S\~ao Paulo\thanksmark{t1} and Universidade Federal da Para\'iba\thanksmark{m2}}

\address{Departamento de Matem\'atica Aplicada e Estat\'istica\\ Universidade de S\~ao
Paulo, 13560-970, S\~ao Carlos - SP, Brazil\\
\printead{e1}\\
\phantom{E-mail:\ }}

\address{Departamento de Matem\'atica\\ Universidade Federal da Para\'iba\\
13560-970, Jo\~ao Pessoa - Para\'iba, Brazil\\
\printead{e2}\\
\printead{e3}}
\end{aug}

\begin{abstract}
We introduce a variational theory for processes adapted to the multi-dimensional Brownian motion filtration that provides a differential structure allowing to describe infinitesimal evolution of Wiener functionals at very small scales. The main novel idea is to compute the ``sensitivities'' of processes, namely derivatives of martingale components and a weak notion of infinitesimal generators, via a finite-dimensional approximation procedure based on controlled inter-arrival times and approximating martingales. The theory comes with convergence results that allow to interpret a large class of Wiener functionals beyond semimartingales as limiting objects of differential forms which can be computed path wisely over finite-dimensional spaces. The theory reveals that solutions of BSDEs are minimizers of energy functionals w.r.t Brownian motion driving noise.
\end{abstract}



\end{frontmatter}

\section{Introduction}
Let $(\Omega, \mathcal{F},\mathbb{P})$ be a probability space equipped with a filtration $\mathbb{F} = (\mathcal{F}_t)_{t\ge 0}$ generated by a multi-dimensional noise process $W$. The goal of this work is to present a systematic approach to concretely analyze the infinitesimal variation of a given $\mathbb{F}$-adapted process w.r.t $W$. A similar type of question has been studied over the last four decades by means of Malliavin calculus and White Noise analysis. In this context, infinitesimal variation of smooth random variables (in the sense of Malliavin) and stochastic distributions (in the sense of Hida) are studied without taking into account an underlying filtration, so that, loosely speaking, these approaches are anticipative in nature. In the present paper, we are interested in providing a non-anticipative calculus in order to characterize variational properties of adapted processes and, more importantly, we aim to provide concrete tools to obtain ``maximizers'' of variational problems w.r.t $W$. We also make a special effort to cover the largest possible class of processes adapted to a given $W$.

Recently, a new branch of stochastic calculus has appeared, known as functional
It\^o calculus, which results to be an extension of classical It\^o calculus to non-anticipative functionals
depending on the whole path of a noise $W$ and not only on its current value,
see e.g Dupire \cite{dupire}, Cont and Fourni\'e \cite{cont,cont1},  Cosso and Russo \cite{cosso1,cosso2}, Peng and Song \cite{peng2}, Buckdahn, Ma, and Zhang \cite{buckdhan}, Keller and Zhang \cite{keller}, Ohashi, Shamarova and Shamarov \cite{ohashi} and Oberhauser \cite{ober}. Inspired by Peng \cite{peng}, the issue of providing a suitable definition of path-dependent PDEs has attracted a great interest, see e.g Peng and Wang \cite{peng1}, Ekren, Keller, Touzi and Zhang \cite{touzi1}, Ekren, Touzi and Zhang \cite{touzi2,touzi3}, Ekren and Zhang \cite{ekren}, Cosso and Russo \cite{cosso3} and Flandoli and Zanco \cite{flandoli}. In the present work, we develop a weak version of the functional It\^o calculus inspired by the discretization scheme introduced by Le\~ao and Ohashi \cite{LEAO_OHASHI2013,LEAO_OHASHI2017}.

\subsection{Main setup and contributions}
In this paper, we are interested in developing a general differentiation theory for $\mathbb{F}$-adapted processes where $\mathbb{F}$ is generated by a state noise $W$ which drives the randomness of the system. We are interested in developing a general methodology to compute infinitesimal variations of $X$ w.r.t the underlying state noise $W$ in three fundamental cases:

\begin{enumerate}[label=(\roman*)]
  \item A priori, $X$ does not possess enough regularity w.r.t $W$.
  \item $X$ depends on the whole path of $W$.
  \item Explicit functional non-anticipative representations of $X$ are not available.
\end{enumerate}
Cases (i) and (iii) occur very frequently in stochastic control problems and it has been one of the motivations for the use of viscosity methods in PDEs in the Markovian setup. When situation (ii) takes place, we cannot rely on classical approaches. In order to deal with cases (i), (ii) and (iii), we develop a theory based on suitable discretizations on the level of the noise rather than $X$ itself. The methodology can be interpreted as a weak functional stochastic calculus for path-dependent systems which can be also interpreted path wisely over finite-dimensional spaces by means of a suitable regularization procedure. In contrast to the classical approximation schemes in the literature based on deterministic discretizations on the time scale, we develop a type of space-filtration discretization procedure $(\mathbb{F}^k)_{k\ge 1}$ which allows us to drastically reduce the dimension of the problem.

We choose the underlying state noise $W$ as a $d$-dimensional Brownian motion $B$ and the fundamental objects to be analyzed will be $\mathbb{F}$-adapted processes $X$ that we call as \textit{Wiener functionals}. The theory is designed to analyse the sensitivity of $X$ w.r.t $B$ under rather general assumptions and much beyond the standard literature on smooth pathwise functional calculus and other Sobolev-type formulations. The methodology is based on suitable approximating structures $$\mathcal{Y} = \big((X^k)_{k\ge 1},\mathscr{D}\big)$$
equipped with a continuous-time random walk approximation $\mathscr{D} = \{\mathcal{T},A^{k,j};\\
j=1\ldots,d, k\ge 1\}$ driven by a suitable class of waiting times $\mathcal{T} = \{T^{k,j}_n; 1\le j\le d, n,k\ge 1\}$ (see (\ref{stopping_times}) and (\ref{rw})) which encodes the evolution of the Brownian motion at small scales. The sequence of processes $(X^k)_{k\ge 1}$ has to be interpreted as a model simplification naturally defined on finite-dimensional spaces which allows us to concretely approach variational properties of $X$ w.r.t $B$ by means of suitable derivative operators $\big(\mathcal{D}^{\mathcal{Y},k}X, U^{\mathcal{Y},k}X\big)$, where

\begin{equation}\label{dintrod1}
\mathcal{D}^{\mathcal{Y},k}X\quad \text{encodes the variation of}~X~\text{w.r.t}~B
\end{equation}
 and

\begin{equation}\label{dintrod2}
U^{\mathcal{Y},k}X\quad \text{encodes variations of}~X~\text{``orthogonal'' to}~B.
\end{equation}
The variational operators (\ref{dintrod1}) and (\ref{dintrod2}) are constructed via optional stochastic integration and $\mathbb{F}^k$-dual predictable projections, respectively, in the spirit of Dellacherie and Meyer \cite{dellacherie} at the level of a given probability measure. Due to the nice structure coming from $\mathscr{D}$ which allows us to reduce the dimension, the operators (\ref{dintrod1}) and (\ref{dintrod2}) can also be computed path wisely in a very concrete way.

Conceptually, the approach developed in this article for analyzing the infinitesimal variation of Wiener functionals w.r.t Brownian state at very small scales consists of three steps.

\begin{enumerate}
  \item Based on the available information given by a Wiener functional $X$ at hand, one designs a discrete-type structure $\mathcal{Y} = \big((X^k)_{k\ge 1}, \mathscr{D}\big)$ where $X^k = X^k(\mathscr{D})$ is a pure-jump process driven by a discrete-type skeleton $\mathscr{D}=\{\mathcal{T},A^{k,j};j=1,\ldots,d, k\ge 1\}$. We call $\mathcal{Y} = \big((X^k)_{k\ge 1}, \mathscr{D}\big)$ as an \textit{imbedded discrete structure} associated with $X$ (see Definition \ref{IDS}).
  \item At the level of imbedded discrete structures, one has concrete functionals defined on finite-dimensional spaces where one is able to compute path wisely the variational operators $\big(\mathcal{D}^{\mathcal{Y},k}X, U^{\mathcal{Y},k}X\big)$ of $(X^k)_{k\ge 1}$ w.r.t $\mathscr{D}$ freely. At this point, one tries to obtain as much information as possible of $X$ by computing $\big(\mathcal{D}^{\mathcal{Y},k}X, U^{\mathcal{Y},k}X\big)$.
  \item In a final step, one has to prove the information obtained in step 2 is consistent with $X$. At this stage, one has to prove the $X^k\rightarrow X$ as the level of discretization $k\rightarrow +\infty$ and, eventually if $X$ admits enough regularity, the limits of $\big(\mathcal{D}^{\mathcal{Y},k}X, U^{\mathcal{Y},k}X\big)$ might be used to get limiting representations and a more refined information on $X$.
\end{enumerate}
It is important to emphasize that it is shown that each Wiener functional $X$ is equipped with a canonical imbedded discrete structure $\mathcal{Y} = \big((X^k)_{k\ge 1},\mathscr{D}\big)$ (see (\ref{canonicalST}), Lemma \ref{deltaconvergence} and Corollary \ref{deltaisstable} for the particular case of Dirichlet processes). However, our methodology requires effort on the part of the ``user'' in order to specify the ``good'' structure that is suitable for analyzing a given problem at hand. The theory generalizes other weak formulations \cite{buckdhan,cont2,peng2} and also the pathwise approaches \cite{dupire,cont,cont1,ohashi,cosso1,cosso2} restricted to Brownian states.

The philosophy of this work is \text{not} to propose representations for Wiener functionals (although we present some of them), but rather a concrete way to depict the infinitesimal variation of processes at very small scales via $\big(\mathcal{D}^{\mathcal{Y},k}X, U^{\mathcal{Y},k}X\big)$ for a given choice of structure $\mathcal{Y}$ associated with $X$. In this direction, the article reveals that a large class of Wiener functionals (including Dirichlet processes) can be viewed as limiting objects of differential forms attached to imbedded discrete structures which have to be carefully designed case by case.

In Theorem \ref{towardsD1}, we construct a differential structure for a Wiener functional (see Definition \ref{defweakder}) of the form

\begin{equation}\label{repINTRODUC}
X(t) = X(0) + \sum_{j=1}^d \int_0^tH_j dB^j + V(t);  0\le t\le T
\end{equation}
where $dB^j$ is the It\^o integral, $H=(H_1, \ldots, H_d)$ is an $\mathbb{F}$-adapted square-integrable process and $V$ has continuous paths without a priori regularity conditions. A non-anticipative process $\mathcal{D}^\mathcal{Y}X$ is constructed based on limits of the variations

$$\mathcal{D}^{\mathcal{Y},k,j}X (T^{k,j}_n) = \frac{\Delta X^k(T^{k,j}_n)}{\Delta A^{k,j}(T^{k,j}_n)}; n\ge 1, 1\le j\le d$$
for a given imbedded discrete structure $\mathcal{Y}  = \big((X^k)_{k\ge 1},\mathscr{D}\big)$. It turns out $\mathcal{D}^\mathcal{Y}X = H$ for every structure $\mathcal{Y}$ (up to a stability property) associated with $X$. Then, we shall define $\mathcal{D}X = \mathcal{D}^\mathcal{Y}X$ for every stable imbedded discrete structure $\mathcal{Y}$ (see Theorem \ref{towardsD1} and (\ref{defDERIVATIVE})).

The differential structure of $X$ heavily depends on possibly orthogonal variations (\ref{dintrod2}) of $X$ w.r.t noise and this is encoded by the drift term $V$ in (\ref{repINTRODUC}) as demonstrated by Proposition \ref{diffform}. Theorem \ref{ItoP} and Section \ref{pvarsec} show that distinct classes of processes are clearly distinguished by (\ref{dintrod2}) as $k\rightarrow+\infty$ which allows us to investigate variational properties of Wiener functionals very concretely in applications to control theory and much beyond semimartingales. Theorem \ref{strongDth1} and Section \ref{pvarsec} show the methodology developed in this article applies to very irregular drifts of unbounded variation and it covers, in particular, drifts of finite $p$-variation $(1 \le p < 2 )$.


In \cite{LOS1}, the authors present a universal variational characterization of the drifts of weakly differentiable processes. It is revealed the asymptotic behavior of (\ref{dintrod2}) is always encoded by suitable limits of integral functionals of horizontal-type perturbations and first-order variation w.r.t noise having a two-parameter occupation time process $(s,x)\mapsto\mathbb{L}^{k,x}(s)$ for $\mathscr{D}$ as integrators. The connection between weak differentiability and Brownian local-times are established under finite $(p,q)$-variation regularity in the sense of Young.

\subsection{Applications of the theory}
As a test of the relevance of our methodology, in \cite{LEAO_OHASHI2017.1,LEAO_OHASHI2017.2,bezerra}, we apply the theory of this article to develop a concrete and systematic method of obtaining near-stochastic optimal controls in a fully non-Markovian setting

\begin{equation}\label{opintr}
\arg \max_{\eta\in U^T_0}\mathbb{E}[\xi(X^\eta)],
\end{equation}
where $\xi: C([0,T];\mathbb{R}^n)\rightarrow\mathbb{R}$ is a payoff functional and $\{X^\eta; \eta\in U^T_0\}$ is a family of abstract Wiener functionals parameterized by a set of possibly mutually singular measures $U^T_0$. In \cite{LEAO_OHASHI2017.1,LEAO_OHASHI2017.2}, the authors present a concrete method for computing near-stochastic optimal controls for non-Markovian systems of path-dependent SDEs driven by Gaussian noises (including fractional Brownian motion), where in \cite{LEAO_OHASHI2017.1}, both drift and diffusion components are controlled. Monte Carlo methods are developed in \cite{bezerra} for the particular case of optimal stopping problems. Rather than developing representation results for value functionals, the variational theory developed in this article permits to extract near-optimal controls in (\ref{opintr}) without requiring a priory regularity assumptions on value processes driven by controlled Wiener functionals $\{X^\eta; \eta\in U^T_0\}$. The analysis is made via a maximization procedure based on $U^{\mathcal{Y},k}X$ (or its nonlinear version) for a choice of imbedded structure $\mathcal{Y}$ of the controlled state. In \cite{LEAO_OHASHI2017.2}, the differential operator (\ref{dintrod2}) (a nonlinear version in \cite{LEAO_OHASHI2017.1}) plays the role of a generalized Hamiltonian which yields a feasible construction of near-optimal controls beyond the Markovian case.


In the present article, we also exhibit a novel characterization of solutions of backward SDEs (BSDEs) as solutions of a variational problem w.r.t Brownian motion state. In Theorem \ref{bsdecor}, it is shown that an It\^o process $Y$ is a solution of a BSDE (\ref{BSDE}) with terminal condition $\xi$ and driver $g$ if, and only if,

\begin{equation}\label{minimiINTR}
\mathcal{D}Y~\text{minimizes energy in the sense of}~(\ref{minen})~\text{and}~Y(T)=\xi.
\end{equation}


Although the connection between BSDEs and control problems of the form (\ref{opintr}) is well-known (see e.g \cite{elliot}), to our best knowledge, Theorem \ref{bsdecor} is the first result connecting BSDEs to variational problems w.r.t the Brownian motion driving noise.

The setup of this article is based on a fixed probability measure. The fully non-linear case is partially treated in \cite{LEAO_OHASHI2017.1}, where the authors aggregate imbedded discrete structures parameterized by mutually singular measures arising from non-Markovian stochastic control problems of the form (\ref{opintr}). However, a full-fledged theory for abstract functionals defined up to polar sets is postponed to further investigations.

The remainder of this article is organized as follows. The next section summarizes some useful notations used in this work. Section \ref{skeletonsection} provides the discrete structure that lays the foundation of this work. Section \ref{DIFFSTRUC} presents the differential system associated with the skeleton. Section \ref{differentialSECTION} develops the asymptotic limits of the differential operators presented in Section \ref{DIFFSTRUC} and some examples are discussed.

\subsection{Notation}
Throughout this article, we are going to fix a filtered probability space $(\Omega, \mathbb{F},\mathbb{P})$ which supports a $d$-dimensional Brownian motion $B = (B^1, \ldots, B^d)$ where $\mathbb{F}:=(\mathcal{F}_t)_{t\ge 0}$ is the usual $\mathbb{P}$-augmentation of the filtration generated by $B$ under a fixed probability measure $\mathbb{P}$. For a given terminal time $0< T< \infty$, let $\mathbf{B}^p(\mathbb{F})$ be the Banach space of all $\mathbb{F}$-adapted real-valued c\`adl\`ag processes $X$ such that

$$
\mathbb{E}\sup_{0\le t\le T}|X(t)|^p < \infty,
$$
where $1\le p< \infty$. We denote $\mathbf{H}^p(\mathbb{F})$ as the subset of $\mathbf{B}^p(\mathbb{F})$ which consists of all $\mathbb{F}$ - martingales starting from zero. We shall also equip $\mathbf{B}^p(\mathbb{F})$ with the weak topology $\sigma(\mathbf{B}^p,\text{M}^q)$ where $1\le p,q < \infty$ with $\frac{1}{p}+\frac{1}{q}=1$. Recall that the topological dual $\text{M}^q(\mathbb{F})$ of $\mathbf{B}^p(\mathbb{F})$ is the space of processes $V = (V^{pr}, V^{pd})$ such that

\

\noindent (i) $V^{pr}$ and $V^{pd}$ are right-continuous of bounded variation such that $V^{pr}$ is $\mathbb{F}$
- predictable with $V^{pr}_0=0$ and $V^{pd}$ is $\mathbb{F}$ - optional and purely discontinuous.

\

\noindent (ii) $Var (V^{pd}) + Var(V^{pr})\in L^q(\mathbb{P});~\frac{1}{p}+\frac{1}{q}=1,$

\

\noindent where $Var(\cdot)$ denotes the total variation of a bounded variation process over the interval $[0,T]$. The space $\text{M}^{q}(\mathbb{F})$ has the strong topology given by

$$\|V\|_{M^q}: = \|Var(V^{pr})\|_{L^q} +
\|Var(V^{pd})\|_{L^q}.$$

\noindent The duality pair is given by

$$(V,X):= \mathbb{E}\int_{0}^{T} X(s-) dV^{pr}(s) +
\mathbb{E}\int_0^{T} X(s)dV^{pd}(s);\quad X \in \text{B}^p(\mathbb{F}),$$
where the following estimate holds

$$|(V, X)| \le \|V\|_{M^q} \|X\|_{\text{B}^p},$$
for every $V \in \text{M}^q(\mathbb{F})$, $X \in \mathbf{B}^p(\mathbb{F})$ such that $1\le p < \infty$ and $\frac{1}{p} + \frac{1}{q}=1$. We denote $\sigma(\mathbf{B}^p,\text{M}^q)$ the weak topology of $\text{B}^p(\mathbb{F})$. We refer the reader to~e.g~\cite{dellacherie,LEAO_OHASHI2013} for more details on this topology. We will see that the $\sigma(\mathbf{B}^p,\text{M}^q)$-topology will be quite natural for the asymptotic limits of this article.

Inequalities between random variables are understood in the $\mathbb{P}$-a.s sense. Inequalities between processes are understood up to evanescent sets, otherwise, it will be up to $Leb\times \mathbb{P}$-null sets, where $Leb$ is the Lebesgue measure on $\mathbb{R}_+$. The space of $\mathbb{F}$-adapted processes such that

$$\mathbb{E}\int_0^T |X(t)|^2dt < \infty$$
is denoted by $L^2_a(\mathbb{P}\times Leb)$. If $E$ is a Borel set, then we denote $\mathcal{B}(E)$ as the Borel sigma algebra of $E$. The usual jump of a process is denoted by $\Delta Y(t):=Y(t)-Y(t-)$ where $Y(t-)$ is the left-hand limit of a c\`adl\`ag process $Y$. Moreover, for any two stopping times $S$ and $J$, we denote the stochastic intervals $[[S,J [[:=\{(\omega,t); S(\omega) \le t < J(\omega) \}$, $[[S, S]]:=\{(\omega,t); S(\omega)=t \}$ and so on.

In order to make clear the information encoded by a path $x \in D([0,t];\mathbb{R}^d)$ up to a given time $0\le r\le t$, we denote $x_r:= \{x(s): 0 \leq s \leq r \}$ and the value of $x$ at time $0 \leq u \leq t$ is denoted by $x(u)$. This notation is naturally extended to processes. Sometimes, we need to stress that we are working with a functional representation in the spirit of functional calculus~(see e.g~\cite{dupire,cont}). We denote $D([0,t];\mathbb{R}^d)$ as the linear space of $\mathbb{R}^d$-valued c\`adl\`ag paths on $[0,t]$ and we set $\Lambda:=\{(t,\omega_t); (t,\omega)\in [0,T]\times D([0,T];\mathbb{R}^d)\}$. Similarly, we denote by $C([0,t];\mathbb{R}^d)$ the linear space of continuous paths on $[0,t]$ and we set $\hat{\Lambda}: = \{(t,\omega_t); (t,\omega)\in [0,T]\times C([0,T];\mathbb{R}^d)\}$.

\section{The underlying discrete skeleton}\label{skeletonsection}
The weak functional stochastic calculus developed in this paper will be constructed from a class of pure jump processes driven by suitable waiting times which describe the local behavior of the Brownian motion: We set $T^{k,j}_0:=0$ and

\begin{equation}\label{stopping_times}
T^{k,j}_n := \inf\{T^{k,j}_{n-1}< t <\infty;  |B^{j}(t) - B^{j}(T^{k,j}_{n-1})| = \varepsilon_k\}, \quad n \ge 1,
\end{equation}
where $\sum_{k\ge 1}\epsilon_k^2 < \infty$. For each $j\in \{1,\ldots,d \}$, the strong Markov property yields the family $(T^{k,j}_n)_{n\ge 0}$ is a sequence of $\mathbb{F}$-stopping times where the increments $\{\Delta T^{k,j}_n ; n\ge 1\}$ is an i.i.d sequence with the same distribution as $T^{k,j}_1$. By the Brownian scaling property, $\Delta T^{k,j}_1 = \epsilon^2_k \tau$ (in law) where $\tau$ is an absolutely continuous variable with mean equals one and with all finite moments (see e.g \cite{Burq_Jones2008}). Then, we define

\begin{equation}\label{rw}
A^{k,j} (t) := \sum_{n=1}^{\infty} \epsilon_k \sigma^{k,j}_n1\!\!1_{\{T^{k,j}_n\leq t \}};~t\ge0,~j=1\ldots, d, k\ge 1
\end{equation}
where the size of the jumps $\{\sigma^{k,j}_n; n\ge 1\}$ is given by

$$
\sigma^{k,j}_n:=\left\{
\begin{array}{rl}
1; & \hbox{if} \ \Delta A^{k,j}(T^{k,j}_n) > 0 \\
-1;& \hbox{if} \ \Delta A^{k,j}(T^{k,j}_n)< 0. \\
\end{array}
\right.
$$
One can easily check that $\{\sigma^{k,j}_n; n\ge 1\}$ is an i.i.d sequence of $\frac{1}{2}$-Bernoulli random variables for each $k\ge 1$ and $j=1,\ldots, d$. Moreover, $\{\Delta A^{k,j}(T^{k,j}_n); n\ge 1\}$ is independent from $\{\Delta T^{k,j}_n; n\ge 1\}$ for each $j=1,\ldots, d$ and $k\ge 1$. By construction

\begin{equation}\label{akb}
\sup_{t\ge 0}|A^{k,j}(t) - B^j(t)|\le \epsilon_k~a.s
\end{equation}
for every $k\ge 1$.

Let $\mathbb{F}^{k,j} := \{ \mathcal{F}^{k,j}_t; t\ge 0 \} $ be the natural filtration generated by $\{A^{k,j}(t);  t \ge 0\}$.  One should notice that $\mathbb{F}^{k,j}$ is a filtration of discrete type~(see Section 4 (Chap 11) and Section 5 (Chap 5) in ~\cite{he}) in the sense that

\[
\mathcal{F}^{k,j}_t = \Big\{\bigcup_{\ell=0}^\infty D_\ell \cap \{T^{k,j}_{\ell} \le t < T^{k,j}_{\ell+1}\}; D_\ell\in \mathcal{F}^{k,j}_{T^{k,j}_\ell}~\text{for}~\ell \ge 0 \Big\},~t\ge 0,
\]
where

$$\mathcal{F}^{k,j}_{T^{k,j}_m}=\sigma(T^{k,j}_1, \ldots, T^{k,j}_m, \Delta A^{k,j}(T^{k,j}_1), \ldots, \Delta A^{k,j}(T^{k,j}_m))$$
and $\mathcal{F}^{k,j}_0 = \{\Omega, \emptyset \}$ and for $m\ge 1$ and $j=1,\ldots, d$. From Th 5.56 in \cite{he}, we know that

$$\mathcal{F}^{k,j}_{T^{k,j}_m}\cap \{T^{k,j}_m \le t < T^{k,j}_{m+1}\} = \mathcal{F}^{k,j}_t\cap \{T^{k,j}_m \le t < T^{k,j}_{m+1}\},$$
for each $m\ge 0$ and $j=1,\ldots, d$. In this case, $\mathbb{F}^{k,j}$ is a jumping filtration in the sense of Jacod and Skorohod \cite{jacod}. From Th 5.52 in \cite{he}, $\mathbb{F}^{k,j}$ is right-continuous. In the sequel, we denote $F_k$ as the distribution function of $\Delta T^{k,1}_1$ and $f_k := F'_k$.
\begin{lemma}\label{angleLemma}
The process $A^{k,j}$ is a square-integrable $\mathbb{F}^{k,j}$-martingale on $[0,T]$ and its angle bracket is

\begin{equation}\label{angleformula}
\langle A^{k,j},A^{k,j}\rangle (t)=
\epsilon_k^2\int_0^t h^{k,j}(s)ds; 0\le t\le T,
\end{equation}
where

$$h^{k,j}(s): =  \sum_{n=0}^\infty \frac{f_k(s-T^{k,j}_n)}{1-F_k(s-T^{k,j}_n)}\mathds{1}_{\{T^{k,j}_n < s\le T^{k,j}_{n+1}\}}; 0\le s\le T,$$
for $j=1,\ldots, d$.
\end{lemma}
\begin{proof}
See Section \ref{AppendixangleLemma} (Appendix) for the proof of this result.
\end{proof}
In the sequel, we write $\vee_{i=1}^m a_i: = \max\{a_1, \ldots, a_m\}$ for $(a_i)_{i=1}^m\in \mathbb{R}^m; m\ge 1$.

\begin{lemma}\label{meshlemma}
For every $q\ge 1$ and $\alpha\in (0,1)$, there exists a constant $C$ which depends on $q\ge 1$ and $\alpha$ such that
$$
\mathbb{E}|\vee_{n\ge 1}\Delta T^{k}_n|^q 1\!\!1_{\{T^k_n\le T\}} \le C \Big(\epsilon^{2q} \lceil\epsilon^{-2}_k T\rceil^{(1-\alpha)}\Big)
$$
for every $k\ge 1$.
\end{lemma}
\begin{proof}
In the sequel, we denote $\tau:=\inf\{t\ge 0; |W(t)| = 1\}$ for a real-valued Brownian motion $W$ and let $f_\tau$ be the density of $\tau$. From Lemma 3 in \cite{Burq_Jones2008}, we know that $f_\tau(x) = o (e^{-(\gamma-\varepsilon)})$ as $x\rightarrow +\infty$ for $\gamma=\frac{\pi^2}{8}$ and any $\varepsilon < \gamma$. In this case, one can check there exists $\lambda > 0$ such that the Cramer condition holds

$$\varphi(\lambda):=\mathbb{E}\exp(\lambda \tau) <\infty.$$
Let $\psi(\lambda): = \text{ln}~\varphi(\lambda)$ defined on $\{\lambda \in \mathbb{R}; \varphi(\lambda) < \infty\}$ and the Cramer transform is defined by $H(a):=\sup_{\lambda < 0}[\lambda a - \psi(\lambda)];a< 1 = \mathbb{E}\tau$. We recall that $H(a) > 0$ for every $a< 1$.
To keep notation simple, we set $\gamma_k(t) = \lceil \epsilon^{-2}_k t\rceil $. Let $N^{k,j}(t) := \max\{n; T^{k,j}_n\le t\}; t\ge 0$ be the clock process associated with $A^{k,j}$, let $\tau_n:=\inf\{t >\tau_{n-1}; |W(t) - W(\tau_{n-1})|=1\}; n\ge 1$, $\tau_0:=0$ and $\tau=\tau_1$. Below $C$ is a positive constant which may defer from line to line. By the very definition,

\begin{eqnarray*}
\vee_{n=1}^{N^{k,j}(T)}\Delta T^{k,j}_n &=&\Big(\vee_{n=1}^{N^{k,j}(T)}\Delta T^{k,j}_n\Big)1\!\!1_{\{N^{k,j}(T) < 2\gamma_k(T)\}}\\
&+& \Big(\vee_{n=1}^{N^{k,j}(T)}\Delta T^{k,j}_n \Big)1\!\!1_{\{N^{k,j}(T) \ge 2\gamma_k(T)\}}.
\end{eqnarray*}
Hence,
\begin{eqnarray*}
\mathbb{E}\Big|\vee_{n=1}^{N^{k,j}(T)}\Delta T^{k,j}_n\Big|^q &\le& 2 C \mathbb{E}\Big|\vee_{n=1}^{2\gamma_k(T)} \Delta T^{k,j}_n\Big|^q + C T^q\mathbb{P}\{N^{k,j}(T)\ge 2 \gamma_k(T)\}\\
& &\\
&=:&I^{k}_1 + I^{k}_2.
\end{eqnarray*}

Let us fix $\alpha \in (0,1)$. Take $n=2\gamma_k(T)$ in Lemma \ref{l.estimativas} (see Appendix) and notice that
$$\left(\mathbb{E}[(\Delta T^{k,j}_n)^{q/(1-\alpha)}]\right)^{1-\alpha} = \epsilon_k^{2q}\left(\mathbb{E}[\tau^{q/(1-\alpha)}]\right)^{1-\alpha}=: C\epsilon_k^{2q},$$
where $C$ is a constant depending on $\alpha$ and $q$. Therefore, by applying Lemma \ref{l.estimativas}, we have
$$I^{k}_1\le 2^{-\alpha}C \epsilon^{2q}_k \gamma_k(T)^{(1-\alpha)}; k\ge 1.$$
By the scale invariance, we know that $\Delta T^{k,j}_n$ has the same law of $\epsilon^2_k(\tau_n-\tau_{n-1})$ for every $n\ge 1$. In this case,

$$\mathbb{P}\{N^{k,j}(T)\ge 2\gamma_k(T)\} = \mathbb{P}\Big\{T^{k,j}_{2 \gamma_k(T)} \le T\Big\} = \mathbb{P}\Big\{\epsilon^2_k \tau_{2\gamma_k(T)}\le T\Big\}.$$

In order to evaluate the second term, we notice we shall write $\tau_{2\gamma_k(T)} = \sum_{n=1}^{2\gamma_k(T)}(\tau_n - \tau_{n-1})$ as a sum of an i.i.d sequence with expectation equals to one. By writing,

$$\epsilon^2_k \tau_{2\gamma_k(T)} = \epsilon^2_k 2\gamma_k(T) \frac{\tau_{2 \gamma_k(T)}}{2\gamma_k(T)},
$$
and noticing that $2 \epsilon^2_k \gamma_k(T)  \ge 2T$,
we shall apply classical large deviation theory to get

\begin{eqnarray*}
\mathbb{P}\Big\{ \epsilon^2_k \tau_{2\gamma_k(T)} \le T \Big\} &=& \mathbb{P}\Bigg\{ \frac{\tau_{2 \gamma_k(T)}}{2 \gamma_k(T)} \le \frac{T}{\epsilon^2_k2\gamma_k(T)} \Bigg\} \le \mathbb{P}\Bigg\{ \frac{\tau_{2 \gamma_k(T)}}{2\gamma_k(T)}\le \frac{1}{2} \Bigg\}\\
&\le& \exp\big(-2\gamma_k(T) H(1/2)\big),
\end{eqnarray*}
for every $k\ge 1$. By noticing that $\vee_{n=1}^\infty\Delta T^k_n1\!\!1_{\{T^k_n\le T\}} \le \vee_{n=1}^\infty\Delta T^{k,j}_n1\!\!1_{\{T^{k,j}_n\le T\}}$ a.s for $j=1,\ldots, d$ and summing up the above estimates, we arrive at

$$\mathbb{E}|\max_{n\ge 1}\Delta T^{k}_n|^q 1\!\!1_{\{T^{k}_n \le T\}}\le C \Big(\epsilon^{2q} \gamma_k(T)^{(1-\alpha)} + \exp(-2\gamma_k(T) H(1/2))\Big)$$
for every $k\ge 1$, where $C$ is a constant which depends on $\mathbb{E}|\tau|^{q/1-\alpha}$. This concludes the proof.
\end{proof}

The multi-dimensional filtration generated by $A^k$ is naturally characterized as follows. Let $\mathbb{F}^k := \{\mathcal{F}^k_t ; 0 \leq t <\infty\}$ be the product filtration given by $\mathcal{F}^k_t := \mathcal{F}^{k,1}_t \otimes\mathcal{F}^{k,2}_t\otimes\cdots\otimes\mathcal{F}^{k,d}_t$ for $t\ge 0$. Let $\mathcal{T}:=\{T^k_m; m\ge 0\}$ be the order statistics obtained from the family of random variables $\{T^{k,j}_\ell; \ell\ge 0 ;j=1,\ldots,d\}$. That is, we set $T^k_0:=0$,

$$
T^k_1:= \inf_{1\le j\le d}\Big\{T^{k,j}_1 \Big\},\quad T^k_n:= \inf_{\substack {1\le j\le d\\ m\ge 1} } \Big\{T^{k,j}_m ; T^{k,j}_m \ge  T^k_{n-1}\Big\}
$$
for $n\ge 1$. We observe that the independence among the Brownian motions allows us to state that the elements of $\{T^k_n; n\ge 0\}$ are almost surely distinct for every $k \ge 1$.

\begin{lemma}\label{jfiltration}
For each $k\ge 1$, $(T^k_n)_{n\ge 0}$ is a sequence of $\mathbb{F}^k$-stopping times such that $T^k_n <\infty$~a.s for every $k,n\ge 1$ and $T^k_n \uparrow +\infty~a.s$ as $n\rightarrow \infty$. Moreover, the filtration $\mathbb{F}^k$ satisfies

\begin{equation}\label{FILTRjump}
\mathcal{F}^k_t \cap\{T^k_n \le t < T^k_{n+1}\} = \mathcal{F}^k_{T^k_n}\cap \{T^k_n \le  t < T^k_{n+1}\}; t\ge 0
\end{equation}
where $\mathcal{F}^k_{T^k_n} = \sigma(A^{k,j}(s\wedge T^k_n); s\ge 0, 1\le j\le d)$ for each integer $n\ge 0$.
\end{lemma}
\begin{proof}
See Section \ref{Appendixjfiltration} (Appendix) for the proof of this result.
\end{proof}

Let $\mathcal{F}^{k}_\infty$ be the completion of $\sigma(A^{k,j}(s); s\ge 0; j=1,\ldots,d)$ and let $\mathcal{N}_{k}$ be the $\sigma$-algebra generated by all $\mathbb{P}$-null sets in $\mathcal{F}^{k}_\infty$. With a slight abuse of notation, we write $\mathbb{F}^{k} = (\mathcal{F}^{k}_t)_{t\ge 0}$, where $\mathcal{F}^{k}_t$ is the usual $\mathbb{P}$-augmentation (based on $\mathcal{N}_k$) satisfying the usual conditions. From (\ref{akb}) and Lemma 2.1 in \cite{LEAO_OHASHI2013}, we do have

\begin{equation}\label{weakfiltration}
\lim_{k\rightarrow \infty}\mathbb{F}^k = \mathbb{F}.
\end{equation}
weakly (in the sense of \cite{coquet1}) over $[0,T]$. Moreover, since $\sum_{k\ge 1}\epsilon^2_k < \infty$, then we shall repeat the same arguments given in the proof of Lemma 2.2 in \cite{koshnevisan} to state that

\begin{equation}\label{uniforTk}
\lim_{k\rightarrow+\infty}\sup_{0\le t\le T}|T^{k,j}_{\lceil\epsilon^{-2}_kt\rceil} - t|=0
\end{equation}
almost surely and in $L^2(\mathbb{P})$ for each $j=1,\ldots, d$.

\begin{definition}
The structure $\mathscr{D} = \{\mathcal{T}, A^{k,j}; k\ge 1, 1\le j\le d\}$ is called a \textbf{discrete-type skeleton} for the Brownian motion.
\end{definition}
For a given choice of discrete-type skeleton $\mathscr{D}$, we will construct a differential theory based on functionals written on $\mathscr{D}$. Before we proceed, it is important to point out that there exists a pathwise description of the dynamics generated by a discrete-type skeleton.
\subsection{Pathwise dynamics of the skeleton}\label{pathsection}
Let us define

$$\mathbb{I}_k:=\Big\{ (i^k_1, \ldots, i^k_d); i^k_\ell\in \{-1,0,1\}~\forall \ell \in \{1,\ldots, d\}~\text{and}~\sum_{j=1}^d|i^k_j|=1   \Big\}$$
and $\mathbb{S}_k:=(0,+\infty)\times \mathbb{I}_k$. Let us define $\aleph: \mathbb{I}_k\rightarrow \{1,2,\dots,d\}\times\{-1,1\}$ by

\begin{equation}\label{alephamap}
\aleph(\tilde{i}^{k}):=\big(\aleph_1(\tilde{i}^{k}),\aleph_2(\tilde{i}^{k})\big):=(j,r),
\end{equation}
where $j\in\{1,\dots,d\}$ is the coordinate of $\tilde{i}^k\in \mathbb{I}_k$ which is different from zero and $r\in\{-1,1\}$ is the sign of $\tilde{i}^k$ in the coordinate $j$.

The $n$-fold Cartesian product of $\mathbb{S}_k$ is denoted by $\mathbb{S}_k^n$ and a generic element of $\mathbb{S}^n_k$ will be denoted by $\mathbf{b}^k_n := (s^k_1,\tilde{i}^k_1, \ldots, s^k_n, \tilde{i}^k_n)\in \mathbb{S}^n_k$ where $(s^k_r,\tilde{i}^k_r)\in (0,+\infty)\times \mathbb{I}_k$ for $1\le r\le n$. Let us define $\eta^k_n:=(\eta^{k,1}_n, \ldots, \eta^{k,d}_n)$, where

$$
\eta^{k,j}_n:=\left\{
\begin{array}{rl}
1; & \hbox{if} \  \Delta A^{k,j} (T^k_n)>0 \\
-1;& \hbox{if} \  \Delta A^{k,j} (T^k_n)< 0 \\
0;& \hbox{if} \ \Delta A^{k,j} (T^k_n)=0.
\end{array}
\right.
$$
Let us define
$$\mathcal{A}^k_n:= \Big(\Delta T^k_1, \eta^k_1, \ldots, \Delta T^k_n, \eta^k_n\Big)\in \mathbb{S}^n_k~a.s$$
One should notice that $$\mathcal{F}^k_{T^k_n} = (\mathcal{A}^k_n)^{-1}(\mathcal{B}(\mathbb{S}^n_k)),$$
up to null sets in $\mathcal{F}^k_\infty$, where $\mathcal{B}(\mathbb{S}^k_n)$ is the Borel sigma algebra generated by $\mathbb{S}^n_k; n\ge 1$.

The law of the system will evolve according to the following probability measure defined by

$$\mathbb{P}^k_r(E):=\mathbb{P}\{\mathcal{A}^k_r\in E\}; E\in \mathcal{B}(\mathbb{S}^r_k),$$
for $k,r\ge 1$. By the very definition,

$$\mathbb{P}^k_{n}(\cdot) = \mathbb{P}^k_{r}(\cdot\times \mathbb{S}^{r-n}_k)$$
for any $r> n\ge 1$.
By construction, $\mathbb{P}^k_{r}(\mathbb{S}^{n}_k\times \cdot)$ is a regular measure and $\mathcal{B}(\mathbb{S}_k)$ is countably generated, then it is known (see e.g III. 70-73 in~\cite{dellacherie2}) there exists ($\mathbb{P}^k_{n}$-a.s unique) a disintegration $\nu^k_{n,r}: \mathcal{B}(\mathbb{S}^{r-n}_k)\times\mathbb{S}^{n}_k\rightarrow[0,1]$ which realizes

$$\mathbb{P}^k_{r}(D) = \int_{\mathbb{S}^{n}_k}\int_{\mathbb{S}^{r-n}_k} 1\!\!1_{D}(\textbf{b}^k_{n},q^k_{n,r})\nu^k_{n,r} (dq^k_{n,r}|\textbf{b}^k_{n})\mathbb{P}^k_{n}(d\textbf{b}^k_{n})$$
for every $D\in \mathcal{B}(\mathbb{S}^{r}_k)$, where $q^k_{n,r}$ is the projection of $\textbf{b}^k_r = (s^k_1,\tilde{i}^k_1, \ldots, s^k_r,\tilde{i}^k_r)\in \mathbb{S}^r_k$ onto the last $(r-n)$ components, i.e., $q^k_{n,r} = (s^k_{n+1},\tilde{i}^k_{n+1}, \ldots,s^k_{r},\tilde{i}^k_{r} )$. If $r=n+1$, we denote $\nu^k_{n+1}:=\nu^k_{n,n+1}$. By the very definition, for each $E\in \mathcal{B}(\mathbb{S}_k)$ and $\mathbf{b}^k_{n}\in \mathbb{S}_k^{n}$, we have

$$
\nu^k_{n+1}(E|\mathbf{b}^k_{n})= \mathbb{P}\Big\{(\Delta T^k_{n+1}, \eta^k_{n+1})\in E|\mathcal{A}^k_{n} = \mathbf{b}^k_{n}\Big\}; n\ge 1.
$$
\section{Abstract Differential Skeleton}\label{DIFFSTRUC}
In this section, we present a differential structure imbedded into the Brownian motion state variable based on a discrete-type structure $\mathscr{D}$.

\begin{definition}
A \textbf{Wiener functional} is an $\mathbb{F}$-adapted continuous process which belongs to $\mathbf{B}^2(\mathbb{F})$.
\end{definition}

\begin{definition}\label{GASdef}
We say that a pure jump $\mathbb{F}^k$-adapted process of the form

$$
X^k(t) = \sum_{n=0}^\infty X^{k}(T^k_n)1\!\!1_{\{T^k_n\le t < T^k_{n+1}\}}; 0\le t\le T,
$$
is a \textbf{good approximating sequence} (henceforth abbreviated by GAS) w.r.t $X$ if $\mathbb{E}[X^k, X^k|(T)<\infty$ for every $k\ge 1$ and

$$\lim_{k\rightarrow+\infty}X^k=X \quad \text{weakly in}~\mathbf{B}^2(\mathbb{F}).$$
\end{definition}
\begin{definition}\label{IDS}
An \textbf{imbedded discrete structure} $\mathcal{Y} = \big( (X^k)_{k\ge 1}, \mathscr{D}\big)$ for a Wiener functional $X$ consists of the following elements:
\begin{itemize}
  \item A discrete-type skeleton $\mathscr{D}=\{\mathcal{T}, A^{k,j}; k\ge 1, 1\le j\le d\}$ for the Brownian state $B$.
  \item A GAS $\{X^k; k\ge 1\}$ w.r.t $X$ associated with the above discrete-type skeleton.
\end{itemize}
\end{definition}

Next, we show there exists a canonical way to attach an imbedded discrete structure to an arbitrary Wiener functional $X$.

\

\noindent \textbf{Canonical imbedded discrete structure:} One typical example of an imbedded discrete structure $\mathcal{Y} = \big( (\delta^kX)_{k\ge 1}, \mathscr{D}\big)$ for a Wiener functional $X$ is given by

\begin{equation}\label{canonicalST}
\delta^kX(t) :=\sum_{n=0}^\infty \mathbb{E}\big[X(T^k_n)|\mathcal{F}^k_{T^k_n}\big]\mathds{1}_{\{T^k_n\le t < T^k_{n+1}\}}; ~0\le t\le T.
\end{equation}

\begin{lemma}\label{deltaconvergence}
If $X$ is a Wiener functional, then $\mathcal{Y} = \big( (\delta^kX)_{k\ge 1}, \mathscr{D}\big)$ is an imbedded discrete structure for $X$.
\end{lemma}
\begin{proof}
Let us denote $X^k:=\sum_{n=0}^\infty X(T^k_n) 1\!\!1_{[[T^k_n , T^k_{n+1}[[}$. Triangle inequality yields

\begin{eqnarray}
\nonumber|\delta^kX(t) - X(t)|&\le& |\delta^kX(t) - X^k(t)| + |X^k(t) - X(t)|\\
\nonumber&\le& \sup_{n\ge 1}\big|\mathbb{E}[X(T^k_n)|\mathcal{F}^k_{T^k_n}] - X(T^k_n)| \big| 1\!\!1_{\{T^k_n\le T\}}\\
\label{deltain}&+& \sup_{0\le t\le T}|X^k(t) - X(t)|\\
\nonumber&=:&J^k_1 + J^k_2; 0\le t \le T.
\end{eqnarray}
By construction $\mathbb{F}^k\subset \mathbb{F}; k\ge 1$ so that we may apply Th. 1 in \cite{coquet1} to safely state that $\lim_{k\rightarrow\infty} J^k_1=0$ in probability. By the very definition, for a given $\epsilon>0$

$$\{\sup_{0\le t\le T}|X^k(t) - X(t)|> \epsilon\} = \{\sup_{0\le t\le T}\max_{n\ge 1}|X^k(T^k_n) - X(t)| 1\!\!1_{\{T^k_n\le t < T^k_{n+1}\}}> \epsilon\}.$$
Lemma \ref{meshlemma} and the path continuity of $X$ allow us to conclude $\lim_{k\rightarrow\infty} J^k_2=0$ in probability. From (\ref{deltain}), we shall apply Doob's maximal inequality on the discrete-time martingale $\mathbb{E}[\sup_{0\le t\le T}|X(t)||\mathcal{F}^k_{T^k_n}]; n\ge 1$ to get
$$\sup_{k\ge 1}\mathbb{E}\sup_{0\le t\le T}|\delta^kX(t) - X(t)|^2\le C\mathbb{E}\sup_{0\le t\le T}|X(t)|^2< \infty$$
and hence $\lim_{k\rightarrow \infty}\delta^kX = X$ strongly in $\textbf{B}^1(\mathbb{F})$. Lastly, $\{\delta^kX;k\ge 1\}$ is bounded in $L^2(\Omega; E)$ where $E$ is the Banach space of c\`adl\`ag functions from $[0,T]$ to $\mathbb{R}$ equipped with the sup norm. Since $L^2(\Omega;\mathbb{R})$ is reflexive, then we shall apply Th 2.1 and Corollary 3.3 in \cite{diestel} to state that $\{\delta^kX; k\ge 1\}\subset \textbf{B}^2 \subset L^2(\Omega; E)$ is weakly-relatively compact w.r.t $L^2(\Omega; E)$-topology. Since $\textbf{B}^2\subset L^2(\Omega; E)$ ($\textbf{B}^2$ been closed) and $\lim_{k\rightarrow \infty}\delta^kX = X$ stronlgy in $\textbf{B}^1$, then all $\textbf{B}^2$-weak limit points of $\{\delta^kX; k\ge 1\}$ are equal and hence we do have weak convergence in $\textbf{B}^2$.
\end{proof}

\

\noindent \textbf{Functional imbedded discrete structures:}  Another example of an imbedded discrete structure can be constructed starting with a fixed non-anticipative functional representation. In the sequel, we make use of the following notation

$$\omega_t: = \omega(t\wedge \cdot); \omega \in D([0,T];\mathbb{R}^d).$$

This notation is naturally extended to processes. We say that $F$ is a \textit{non-anticipative} functional if it is a Borel mapping and

$$F_t(\omega) = F_t(\omega_t); (t,\omega)\in[0,T]\times D([0,T];\mathbb{R}^d).$$
We recall the set $\Lambda=\{(t,\omega_t); t\in [0,T]; \omega\in D([0,T];\mathbb{R}^d)\}$. Let us endow $\Lambda$ with the metric

$$\textbf{d}((t,\omega); (t',\omega')): = \sup_{0\le u\le T}\|\omega(u\wedge t) - \omega'(u\wedge t')\|_{\mathbb{R}^d} + |t-t'|.$$
Let $X$ be a Wiener functional and let $\mathcal{Y}= \big((X^k)_{k\ge 1},\mathscr{D}\big)$ be an imbedded structure for $X$. By Doob-Dynkin lemma, there exists a functional $\hat{F}$ defined on $\hat{\Lambda} = \{(t,\omega_t); t\in [0,T]; \omega\in C([0,T];\mathbb{R}^d)\}$ such that

\begin{equation}\label{funcrep}
\hat{F}_t(B_t)=  X(t),~0\le t\le T.
\end{equation}
When we write $X=F(B)$ for a given non-anticipative functional $F$ defined on $\Lambda$ it is implicitly assumed that we are fixing a functional $F$ which is consistent to $\hat{F}$ in the sense that
$F_t(x_t) = \hat{F}_t(x_t)$ for every $x\in C([0,T];\mathbb{R}).$ Let $X$ be a Wiener functional

$$
 X(t)= F_t(B_t);~0\le t\le T,
$$
where a $F$ is a non-anticipative functional $F$ defined on $\Lambda$. Then, we shall define the following structure $\mathcal{F}:=\big((\textbf{F}^k)_{k\ge 1}, \mathscr{D}\big)$,

\begin{equation}\label{fraaa}
\textbf{F}^k(t):=\sum_{\ell=0}^\infty F_{T^k_\ell}(A^k_{T^k_\ell})1\!\!1_{\{T^{k}_{\ell}\le t < T^{k}_{\ell + 1}\}},~0\le t\le T.
\end{equation}
The reader should not confuse $F(A^k)$ with $\textbf{F}^k$ because $\{\textbf{F}^k(t);0\le t\le T\}$ is a pure jump process while $\{F_t(A^k_t);0\le t\le T\}$ does not necessarily has this property. Under continuity assumptions in the sense of pathwise functional calculus (see \cite{dupire,cont}), one can easily check that $\lim_{k\rightarrow \infty}\textbf{F}^k=F(B)$ weakly in $\mathbf{B}^2(\mathbb{F})$ so that $\mathcal{F}=\big((\textbf{F}^k)_{k\ge 1}, \mathscr{D}\big)$ is an imbedded discrete structure for the Wiener functional $F(B)$.

Concrete examples of imbedded discrete structures arise in many contexts:

\begin{itemize}
  \item Discretization of value processes arising from path-dependent optimal stochastic control problems. See the works \cite{bezerra,LEAO_OHASHI2017.1,LEAO_OHASHI2017.2}.
  \item Euler-Maruyama schemes arising from path-dependent stochastic differential equations driven by Gaussian noises. See the works \cite{LEAO_OHASHI2017.1,LEAO_OHASHI2017.2}.
  \item Functional imbedded discrete structures associated with path-dependent functionals under $(p,q)$-variation regularity. See \cite{LOS1}.
\end{itemize}
In principle, one can always construct an imbedded discrete structure to a Wiener functional $X$ by only observing the basic probabilistic structure of $X$. The use of functional structures $\mathcal{F}$ is only indicated when one has some a priori information on a functional which realizes $X=F(B)$. In general, this is not the case and other types of imbedded structures must be considered. See \cite{LEAO_OHASHI2017.1,LEAO_OHASHI2017.2} for concrete examples of applications to control theory.

\

\noindent \textbf{Differential operators on imbedded discrete structures}. The reader should really think an imbedded discrete structure as a model simplification for a given Wiener functional $X$ where we are able to compute freely the sensitivities of $X$ w.r.t the Brownian state, i.e., without any regularity assumptions. In the sequel, we provide a detailed explanation on this point. For a given Wiener functional $X$, let us choose an imbedded discrete structure $\mathcal{Y} = \big( (X^k)_{k\ge 1}, \mathscr{D}\big)$ associated with $X$.

In the sequel, $\big(Y\big)^{p,k}$ denotes the $\mathbb{F}^k$-dual predictable projection of an $\mathbb{F}^k$-adapted process $Y$ with locally integrable variation (see Chap.5 in \cite{he}). At first, we observe that the $\mathbb{F}^k$-dual predictable projection $\big(X^k - X^k(0)\big)^{p,k}$ of the process $X^k - X^k(0)$ is well-defined and it is the unique $\mathbb{F}^k$-predictable bounded variation process such that

$$
X^k -X^k(0) - \big(X^k-X^k(0)\big)^{p,k} ~\text{is an}~\mathbb{F}^k-\text{local martingale}.
$$
Then, one can write

\begin{equation}\label{deltaXdef}
X^{k} (t) = X^k(0) + \sum_{0< s\le t}\Delta X^k(s)  = X^k(0) + \sum_{j=1}^d \int_{0}^t \mathcal{D}^{\mathcal{Y},k,j}X(u) dA^{k,j} (u)
\end{equation}
where

$$
\mathcal{D}^{\mathcal{Y},k,j}X(u) :=  \sum_{\ell=1}^{\infty} \frac{\Delta X^{k} (T^{k,j}_\ell)}{\Delta A^{k,j}(T^{k,j}_\ell)} 1\!\!1_{\{T^{k,j}_\ell=u\}}; 0\le u\le T,\quad \mathcal{Y} = \big((X^k)_{k\ge 1},\mathscr{D}\big),
$$
and the integral in \eqref{deltaXdef} is interpreted in the Lebesgue-Stieltjes sense. In the sequel, $\mu_{[A^{k,j}]}$ is the Dol\'eans measure (see e.g Chap.5 in \cite{he}) generated by the point process $[A^{k,j},A^{k,j}]; 1\le j\le d, k\ge 1$, i.e.,

$$\mu_{[A^{k,j}]}(H):=\mathbb{E}\int_0^T\mathds{1}_{H}(s)d[A^{k,j},A^{k,j}](s); H\in \mathcal{F}^k_T\times \mathcal{B}([0,T]).$$
From the square integrability of the martingale $A^{k,j}$, $\mu_{[A^{k,j}]}$ is a finite measure. In the sequel, $\mathcal{P}^k$ is the $\mathbb{F}^k$-predictable sigma-algebra of $[0,T]\times \Omega$. We observe that

$$\int_0^\cdot\frac{\mathcal{D}^{\mathcal{Y},k,j}X}{\Delta A^{k,j}}d [A^{k,j},A^{k,j}]$$
is a process with locally integrable variation. Then, there exists a unique (up to sets of $\mu_{[A^{k,j}]}$-measure zero) $\mathbb{F}^k$-predictable process $\mathbb{E}_{\mu_{[A^{k,j}]}}\Big[\frac{\mathcal{D}^{\mathcal{Y},k,j}X}{\Delta A^{k,j}}\big|\mathcal{P}^k\Big]$ such that

$$\Big(\int_0^\cdot\frac{\mathcal{D}^{\mathcal{Y},k,j}X}{\Delta A^{k,j}}d [A^{k,j},A^{k,j}]\Big)^{p,k} = \int_0^\cdot \mathbb{E}_{\mu_{[A^{k,j}]}}\Big[\frac{\mathcal{D}^{\mathcal{Y},k,j}X}{\Delta A^{k,j}}\big|\mathcal{P}^k\Big] d\langle A^{k,j},A^{k,j}\rangle.$$
See e.g Th 5.25 and remark in \cite{he}. We then denote

\begin{equation}\label{weakinfG}
U^{\mathcal{Y},k,j}X(u):=\mathbb{E}_{\mu_{[A^{k,j}]}}\Bigg[\frac{\mathcal{D}^{\mathcal{Y},k,j}X}{\Delta A^{k,j}}\Big|\mathcal{P}^k\Bigg](u);~ 0\le u\le T, k\ge 1, 1\le j\le d,
\end{equation}
where it is understood that the stochastic process $\mathcal{D}^{\mathcal{Y},k,j}X/\Delta A^{k,j}$ is null on the complement of the union of stochastic intervals $\cup_{n=1}^\infty [[T^{k,j}_n,T^{k,j}_n]]$. We call

\begin{equation}\label{Fkweakgen}
\sum_{j=1}^d U^{\mathcal{Y},k,j}X
\end{equation}
as the $\mathbb{F}^k$-\textit{weak infinitesimal generator} of the imbedded discrete structure $\mathcal{Y}$ w.r.t $X$. We then denote $\mathcal{D}^{\mathcal{Y},k}X = (\mathcal{D}^{\mathcal{Y},k,1}X, \ldots,\mathcal{D}^{\mathcal{Y},k,d}X)$ and $U^{\mathcal{Y},k}X = (U^{\mathcal{Y},k,1}X, \ldots, U^{\mathcal{Y},k,d}X)$.
\begin{lemma}\label{explicit}
Let $\mathcal{Y} = \big((X^k)_{k\ge 1}, \mathscr{D}\big)$ be an imbedded discrete structure w.r.t $X$. The $\mathbb{F}^k$-dual predictable projection of $X^{k} - X^k(0)$ is given by the absolutely continuous process

$$
\sum_{j =1}^d \int_0^tU^{\mathcal{Y},k,j} X (s) d\langle A^{k,j}, A^{k,j} \rangle (s),\quad 0\le t \le T.
$$
Moreover,

\begin{equation}\label{ger1}
\sum_{j=1}^d U^{\mathcal{Y},k,j}X(T^k_{n+1}) = \mathbb{E}\Bigg[\frac{\Delta X^k(T^k_{n+1})}{\epsilon_k^2}\Big|\mathcal{F}^k_{T^k_{n+1}-}\Bigg]~a.s,
\end{equation}
for each $n\ge 0$ and $k\ge 1$.
\end{lemma}

\begin{proof}
See Section \ref{Appendixexplicit} (Appendix) for the proof of this result.
\end{proof}
For a given embedded discrete structure $\mathcal{Y}= \big((X^k)_{k\ge 1}, \mathscr{D}\big)$ w.r.t $X$, let us denote

\begin{eqnarray*}
\sum_{j=1}^d\oint_0^t\mathcal{D}^{\mathcal{Y},k,j}X(s)dA^{k,j} (s)&:=&\sum_{j=1}^d\int_0^t\mathcal{D}^{\mathcal{Y},k}X(s) dA^{k,j} (s)\\
& &\\
&-&\Big(\sum_{j=1}^d\int_0^{\cdot}\mathcal{D}^{\mathcal{Y},k} X(s) d A^{k,j} (s)\Big)^{p,k}(t),
\end{eqnarray*}
where $\oint$ is the $\mathbb{F}^k$-optional integral as introduced by Dellacherie and Meyer (see Chap 8, section 2 in\cite{dellacherie}) for optional integrands. Let us denote

\begin{equation}\label{uncOPERATORS}
\mathbb{D}^{\mathcal{Y},k,j} X(t) :=  \sum_{\ell=1}^{\infty} \mathcal{D}^{\mathcal{Y},k,j}X(t) 1\!\!1_{\{T^{k}_\ell\le t<  T^{k}_{\ell+1}\}},\quad \mathbb{U}^{\mathcal{Y},k,j}X(t) := U^{\mathcal{Y},k,j}X(t) \frac{d\langle A^{k,j}, A^{k,j}\rangle}{dt}
\end{equation}
for $0\le t\le T$. We then arrive at the following result.

\begin{proposition}\label{repdelta}
Let $\mathcal{Y}= \big((X^k)_{k\ge 1}, \mathscr{D}\big)$ be an imbedded discrete structure for a Wiener functional $X$. Then the $\mathbb{F}^k$-special semimartingale decomposition of $X^{k}$ is given by

\begin{equation}\label{discretediffform}
X^{k}(t) = X^k(0) + \sum_{j=1}^d \oint_0^t \mathbb{D}^{\mathcal{Y},k,j}X(s)dA^{k,j}(s) + \sum_{j =1}^d \int_0^t\mathbb{U}^{\mathcal{Y},k,j} X (s) ds,~0\le t \le T.
\end{equation}
\end{proposition}

At this point, we stress that in many cases playing with the variational operators $(\mathcal{D}^{\mathcal{Y},k,j}X,U^{\mathcal{Y},k,j} X)$ attached to an imbedded discrete structure is good enough to infer non-trivial information on $X$ without computing the limit of $(\mathbb{D}^{\mathcal{Y},k,j}X,\mathbb{U}^{\mathcal{Y},k,j} X)$, which in many cases, it can be problematic due to lack of smoothness of $X$ w.r.t state. See \cite{LEAO_OHASHI2017.1, LEAO_OHASHI2017.2} for details. The differential form (\ref{discretediffform}) describes an imbedded discrete structure for a Wiener functional $X$ and it will be the starting point to analyze the sensitivities of $X$ w.r.t Brownian state under rather weak regularity conditions as demonstrated in Section \ref{differentialSECTION}.

\

\noindent \textbf{Pathwise description of the variational operators attached to $\mathcal{Y}$:} In order to compute the preliminary variational operators $(\mathcal{D}^{\mathcal{Y},k,j}X,U^{\mathcal{Y},k,j} X)$, we can actually proceed path wisely because all the objects are path-dependent functionals of the discrete-type skeleton $\mathscr{D}$. In the sequel, for any $(r,n)$ such that $1\le r\le n$ and $\mathbf{b}^k_n = (s^k_1, \tilde{i}^k_1, \ldots, s^k_n, \tilde{i}^k_n)$, we denote

$$\pi_r(\mathbf{b}^k_n):=(s^k_1, \tilde{i}^k_1, \ldots, s^k_r, \tilde{i}^k_r).$$
Let $F^k_n:\mathbb{S}_k^n\rightarrow\mathbb{R}$ be a sequence of Borel functions. Let us define

$$\nabla_j F^{k}(\mathbf{b}^k_n):=\frac{F^k_n (\textbf{b}^k_n) - F^k_{n-1}(\pi_{n-1}(\textbf{b}^k_{n}))}{\epsilon_k\aleph_2(\tilde{i}^k_n)}\mathds{1}_{\{\aleph_1(\textbf{b}^k_n) = j\}},~\mathbf{b}^k_n\in \mathbb{S}^n_k; n\ge 1,$$
for $j=1,\ldots, d$, and

\begin{equation}\label{pathwiseUcond}
\mathscr{U}F^k(\mathbf{b}^k_n):=\int_{\mathbb{S}_k} \frac{F^k_{n+1}(\mathbf{b}^k_n, s^k_{n+1},\tilde{i}^k_{n+1}) - F^k_{n}(\mathbf{b}^k_n)}{\epsilon^2_k}\nu^k_{n+1}(ds^k_{n+1}d\tilde{i}^k_{n+1}|\mathbf{b}^k_n),
\end{equation}
for $\mathbf{b}^k_n\in \mathbb{S}^n_k, n\ge 0$. We observe that for a sequence of Borel maps $F^k_n:\mathbb{S}_k^n\rightarrow\mathbb{R}$ satisfying

$$X^k(T^k_n) = F^k_n (\mathcal{A}^k_n)~a.s$$
we will get

$$
\nabla_j F^{k} (\mathcal{A}^k_n) = \mathcal{D}^{\mathcal{Y},k,j}X(T^k_n)\mathds{1}_{\{\aleph_1(\eta^k_n)=j\}}~a.s, n\ge 0, j=1,\ldots,d,
$$
and
\begin{equation}\label{pathREP2}
\mathscr{U}F^k(\mathcal{A}^k_n)=\mathbb{E}\Bigg[\sum_{j=1}^d U^{\mathcal{Y},k,j}X(T^k_{n+1})\big|\mathcal{F}^k_{T^k_n}\Bigg] ~a.s
\end{equation}
for each $n\ge 0$. The right-hand side of (\ref{pathREP2}) will be called the $\mathbb{F}^k$-conditional weak infinitesimal generator of $X$ based on a structure $\mathcal{Y}$. It turns out that the operator (\ref{pathwiseUcond}) play the role of the Hamiltonian in a non-Markovian optimal stopping problem (see \cite{LEAO_OHASHI2017.2} and example \ref{exampleOS}). In the sequel, we present two examples related to martingales and stochastic control based on a generic Wiener functional.

\begin{example}
Let $X\in \mathbf{H}^2(\mathbb{F})$ be a martingale with terminal condition $\xi\in L^2(\mathcal{F}_T)$ and $\mathbb{F}$ is the filtration generated by a one-dimensional Brownian motion. A simple imbedded structure associated with $X$ is the following: Let $\xi^k:= G^k_{\gamma(k,T)}\big(\mathcal{A}^k_{\gamma(k,T)}\big)$ be an approximation for $\xi$, i.e., $\lim_{k\rightarrow+\infty}\xi^k = \xi$ in $L^2(\mathbb{P})$, where $\gamma(k,T) = \lceil\epsilon^{-2}_kT\rceil$. Let us define the good approximation sequence as

$$X^k(T^k_n):=\mathbb{E}\big[\xi^k|\mathcal{F}^k_{T^k_n}\big]; 0\le n\le \gamma(k,T),$$
Due to the path continuity of $X$, the weak convergence $\mathbb{F}^k\rightarrow \mathbb{F}$, Lemma \ref{meshlemma} and (\ref{uniforTk}), one can easily check that $\big((X^k)_{k\ge 1}, \mathscr{D}\big)$ is an imbedded discrete structure. A simple computation reveals that $X^k(T^k_n) = F^k_n(\mathcal{A}^k_n)~a.s$ where

$$F^k_n (\mathbf{b}^k_n)= \int_{\mathbb{S}^{r-n}}G^k_{r}(\mathbf{b}^k_n, q^k_{n,r})\nu^k_{n,r}(dq^k_{n,r}|\mathbf{b}^k_n); \mathbf{b}^k_n\in \mathbb{S}^n_k;~0\le n < r,$$
where $r$ is a shorthand notation for $\gamma(k,T)$. Then,

$$
\nabla F^k_n(\mathbf{b}^k_n) = \frac{F^k_n (\mathbf{b}^k_n) - F^k_{n-1}(\pi_{n-1}(\mathbf{b}^k_{n}))}{\epsilon_k\tilde{i}^k_n};~\mathbf{b}^k_n \in \mathbb{S}^n_k, 1\le n\le r,
$$
and
\begin{eqnarray*}
\mathscr{U}F^k_n(\mathbf{b}^k_n)&=&\frac{1}{2}\int_{0}^{\infty} \frac{F^k_{n+1}(\mathbf{b}^k_n, x,1) - F^k_{n}(\mathbf{b}^k_n)}{\epsilon^2_k}f_k(x)dx\\
& &\\
&+&\frac{1}{2}\int_{0}^{\infty} \frac{F^k_{n+1}(\mathbf{b}^k_n, x,-1) - F^k_{n}(\mathbf{b}^k_n)}{\epsilon^2_k}f_k(x)dx
\end{eqnarray*}
for $\mathbf{b}^k_n\in \mathbb{S}^n_k, 0\le n\le r$.
\end{example}

\begin{example}\label{exampleOS}
Let us illustrate the case of the optimal stopping problem. In \cite{LEAO_OHASHI2017.2,bezerra}, the authors propose a systematic way in concretely solving optimal stopping problems based on generic Wiener functionals beyond the Markovian case. Let $S$ be the Snell supermartingale process
\begin{equation}\label{snell}
S (t):= \esssup_{\tau\ge t} \mathbb{E} \left[ Z(\tau)  \mid \mathcal{F}_t \right], \quad 0 \leq t \leq T,
\end{equation}
where esssup is computed over the class of all $\mathbb{F}$-stopping times located on $[t,T]$. It is shown that value processes of the form (\ref{snell}) based on reward continuous functionals $Z$ applied to path-dependent SDEs driven by fractional Brownian motion admit imbedded discrete structures $\mathcal{Y} = \big((S^k)_{k\ge 1},\mathscr{D}\big)$ with pathwise representations $S^k(T^k_n) = \mathbb{V}^k_n(\mathcal{A}^k_n)~a.s; 0\le n< r$ where $\mathbb{V}^k$ solves the nonlinear equation

\begin{eqnarray*}
\max \left\{\mathscr{U}\mathbb{V}^k_i\big(\mathbf{b}^k_i\big);  \gamma^k_{i,r}(\mathbf{b}^k_i) - \mathbb{V}^k_i(\mathbf{b}^k_i) \right\} & = & 0;\quad i=r-1, \ldots, 0, \\
\mathbb{V}^k_{r} (\mathbf{b}^k_r) &=&\gamma^k_{r,r}(\mathbf{b}^k_r); \mathbf{b}^k_r\in\mathbb{S}^r_k,
\end{eqnarray*}
where $\gamma^k_{n,r}$ in an explicit pathwise representation for a simple imbedded discrete structure $\mathcal{Z} = \big( (Z^k)_{k\ge 1}, \mathscr{D}\big)$ associated with $Z$.
\end{example}

\section{Differential structure of Wiener functionals}\label{differentialSECTION}
In this section, we present asymptotic results for the sensitivities associated with imbedded discrete structures as defined in the previous section.

\begin{definition}\label{fenergy}
Let $\mathcal{Y}=\big( (X^k)_{k\ge 1}; \mathscr{D}\big)$ be an imbedded discrete structure for a Wiener functional $X$. We say that $\mathcal{Y}$ has \textbf{finite energy} if

$$\mathcal{E}^{2,\mathcal{Y}}(X):=\sup_{k\ge 1}\mathbb{E}\sum_{n\ge 1}|\Delta X^{k}(T^k_n)|^21\!\!1_{\{T_{n}^k \leq T\}} < \infty.$$
\end{definition}

\begin{definition}\label{deltacov}
Let $\mathcal{Y}=\big( (X^k)_{k\ge 1}; \mathscr{D}\big)$ be an imbedded discrete structure for a Wiener functional $X$. We say that $X$ admits the $\mathcal{Y}$-\textbf{covariation} w.r.t to j-th component of the Brownian motion $B$ if the limit

$$\langle X, B^j\rangle^{\mathcal{Y}}(t):=\lim_{k\rightarrow \infty}[X^{k}, A^{k,j}](t)$$
exists weakly in $L^1(\mathbb{P})$ for each $t\in [0,T]$.
\end{definition}
The idea behind Definition \ref{fenergy} is compactness for an imbedded discrete structure which will allow us to extract convergent subsequences on the components of the special semimartingale decomposition given in Proposition \ref{repdelta}. The role of the $\mathcal{Y}$-covariation is to bring stability to the semimartingale decomposition in Proposition \ref{repdelta}.

It is important to point out that the above properties are only important to get convergence of the underlying differential structure $\big(\mathbb{D}^{\mathcal{Y},k,j}X, \mathbb{U}^{\mathcal{Y},k,j}X;\\ j=1,\ldots, d\big)$ but in typical applications, we do not really need the existence of the limiting differential form to solve concrete problems in e.g non-Markovian stochastic control. See \cite{LEAO_OHASHI2017.1, LEAO_OHASHI2017.2} for all the details.

From Proposition \ref{repdelta}, we know that each imbedded discrete structure $\mathcal{Y}=\big( (X^k)_{k\ge 1}; \mathscr{D}\big)$ carries a sequence of $\mathbb{F}^k$-special semimartingale decompositions

$$X^k(t) = X^k(0)+ M^{\mathcal{Y},k}(t) + N^{\mathcal{Y},k}(t); 0\le t\le T,$$
where $M^{\mathcal{Y},k}$ is an $\mathbb{F}^k$-square-integrable martingale and $N^{\mathcal{Y},k}$ is an $\mathbb{F}^k$-predictable absolutely continuous process.

\begin{definition}
Let $X=X(0) + M + N$ be an $\mathbb{F}$-adapted process such that $M\in \mathbf{H}^2(\mathbb{F})$ and $N\in \mathbf{B}^2(\mathbb{F})$ has continuous paths. An imbedded discrete structure $\mathcal{Y}=\big( (X^k)_{k\ge 1}; \mathscr{D}\}$ for $X$ is said to be \textbf{stable} if $M^{\mathcal{Y},k}\rightarrow M$ weakly in $\mathbf{B}^2(\mathbb{F})$ as $k\rightarrow+\infty$.
\end{definition}

\begin{remark}
By the Burkholder-Davis-Gundy's inequality and the fact that $\Delta X^k = \Delta M^{\mathcal{Y},k}$, we have $\mathcal{E}^{2,\mathcal{Y}}(X)<\infty$ for any stable discrete structure $\mathcal{Y}$ w.r.t $X\in \mathbf{B}^2(\mathbb{F})$.
\end{remark}

Going to the literature on convergence of stochastic processes, we can find some compactness conditions to ensure that a given imbedded structure is stable.

\begin{lemma}\label{findingSTABLE}
Let $X = X(0) + M + N$ be a square-integrable continuous semimartingale where $M\in \mathbf{H}^2(\mathbb{F})$ and $N$ is $\mathbb{F}$-adapted with bounded variation. If $\mathcal{Y} = \big((X^k)_{k\ge 1}, \mathscr{D}\big)$ is an imbedded discrete structure w.r.t $X$ such that $\mathcal{E}^{2,\mathcal{Y}}(X)<\infty$, $\lim_{k\rightarrow+\infty}X^k=X$ in $\mathbf{B}^2(\mathbb{F})$ and

$$\text{Var}\big(N^{\mathcal{Y},k}\big)~\text{is tight in}~\mathbb{R},$$
then $\mathcal{Y}$ is stable, where $\text{Var}$ denotes the first variation of a process over $[0,T]$. Let $X$ be a continuous strong Dirichlet process (see Section \ref{examplesDIFF}) with canonical decomposition $X= X(0) + M + N; M\in \mathbf{H}^2(\mathbb{F})$ and $N$ has $2$-null variation. Let $\mathcal{Y} = \big((X^k)_{k\ge 1}, \mathscr{D}\big)$ be a structure associated with $X$. If for every $\epsilon>0$

$$\lim_{k\rightarrow+\infty}\sup_{\ell\ge 1}\mathbb{P}\Bigg(\sum_{r=1}^\infty \big|N^{\mathcal{Y},\ell}(T^{k}_r) - N^{\mathcal{Y},\ell}(T^{k}_{r-1})\big|^2\mathds{1}_{\{T^{k}_r\le T\}} > \epsilon\Bigg)=0,$$
$\sup_{0\le t\le T}|N^{\mathcal{Y},k}(t)|$ is bounded in probability, $\lim_{k\rightarrow+\infty}X^k=X$ in $\mathbf{B}^2(\mathbb{F})$ and finite energy $\mathcal{E}^{2,\mathcal{Y}}(X)< \infty$ holds, then $\mathcal{Y}$ is stable.
\end{lemma}
\begin{proof}
For the semimartingale case, we just need to apply Th. 11 in \cite{memin} to conclude $M^{\mathcal{Y},k}\rightarrow M$ uniformly in probability as $k\rightarrow+\infty$ and this implies $M^{\mathcal{Y},k}\rightarrow M$ weakly in $\mathbf{B}^2(\mathbb{F})$ as $k\rightarrow +\infty$ so that $\mathcal{Y}$ is stable. Th 2 in \cite{coquet3} allows us to conclude stability for the Dirichlet case up to the fact that our partition is random, but one can easily check that all arguments in the proof of Th 2 in \cite{coquet3} apply to our case as well.
\end{proof}

Stability of an imbedded discrete structure is a natural property since we are interested in analyzing Wiener functionals with at least a non-null martingale component. See also Remark \ref{whystability} for further details. Let us now devote our attention to the study of the asymptotic properties of

$$\mathbb{D}^{\mathcal{Y},k,j}X; j=1,\ldots, d,$$
for an imbedded discrete structure $\mathcal{Y} = \big((X^k)_{k\ge 1}, \mathscr{D}\big)$ w.r.t $X$. We set

\begin{equation}\label{DX}
\mathcal{D}^{\mathcal{Y}}_jX:=\lim_{k\rightarrow+\infty}\mathbb{D}^{\mathcal{Y},k,j}X~\text{weakly in}~L^2_a(\mathbb{P}\times Leb)
\end{equation}
whenever the right-hand side of (\ref{DX}) exists for a given finite-energy embedded structure $\mathcal{Y}$ and, in this case, we write

$$\mathcal{D}^{\mathcal{Y}}X:=\big(\mathcal{D}^{\mathcal{Y}}_1X, \ldots, \mathcal{D}^{\mathcal{Y}}_dX\big).$$

We start the analysis with the following result.
\begin{theorem}\label{towardsD1}
Let $X= X(0) + \sum_{j=1}^d\int H_jdB^j + V$ be an $\mathbb{F}$-adapted process such that $H = (H_1,\ldots, H_d), H_j\in L^2_a(\mathbb{P}\times Leb); j=1,\ldots, d$ and $V\in \mathbf{B}^2(\mathbb{F})$ has continuous paths. Then,

$$\mathcal{D}^\mathcal{Y}X = H$$
for every stable imbedded discrete structure $\mathcal{Y} = \big((X^k)_{k\ge 1}, \mathscr{D}\big)$.
\end{theorem}

At first, we observe Lemma 3.4 in \cite{LEAO_OHASHI2013} holds for the $\mathbb{F}^k$-martingale $A^{k,j}$. Then, we have the following result.
\begin{lemma}\label{lfund1}
Let $H_\cdot = \mathbb{E}[ 1\!\!1_{G}|\mathcal{F}_\cdot]$ and $H^k_\cdot =\mathbb{E}[ 1\!\!1_{ G}|\mathcal{F}^k_\cdot]$ be positive and uniformly integrable martingales w.r.t filtrations $\mathbb{F}$ and $\mathbb{F}^k$, respectively, where $G\in \mathcal{F}_T$. Then, for each $1\le j\le d$,

$$\Bigg\|\int_0^\cdot H(s)dB^j_s -  \oint_0^\cdot H^k(s)dA^{k,j}(s)\Bigg\|_{\mathbf{B}^2}\rightarrow 0\quad \text{as}~k\rightarrow \infty.$$
\end{lemma}
\begin{proof}
Since $A^{k,j}$ is a pure jump martingale and $\mathbb{E}\sup_{0\le t\le T}|B^j(t)|^p< \infty$ for every $p> 2$, then we shall apply Lemma 3.4 in \cite{LEAO_OHASHI2013} to conclude the proof.
\end{proof}

\begin{lemma}\label{lfund2}
Let $\{Y^k;k\ge 1\}$ be a sequence of $\mathbb{F}^k$-square-integrable martingales such that $\lim_{k\rightarrow \infty}Y^{k}=Z$ weakly in $\mathbf{B}^2(\mathbb{F})$, where $Z\in \mathbf{H}^2(\mathbb{F})$. Then, for each $1\le j\le d$,

\begin{equation}\label{conq}
\lim_{k\rightarrow \infty}[Y^{k},A^{k,j}](t) = [Z,B^j](t)\quad\text{weakly in}~L^1(\mathbb{P})
\end{equation}
for every $t\in [0,T]$.
\end{lemma}
\begin{proof}
In the notation of the proof of Lemma 3.5 in~\cite{LEAO_OHASHI2013},we observe that since for every BMO $\mathbb{F}$-martingale $U$, we have $\lim_{k\rightarrow \infty}[Z^{k,X}, U](t)=[Z,U](t)$ weakly in $L^1(\mathbb{P})$ for every $t\in [0,T]$, then we shall take $W=B$. By using Lemma \ref{lfund1}, the proof of Lemma 3.5 in \cite{LEAO_OHASHI2013} works perfectly for the pure-jump sequence $\{Y^k; k\ge 1\}$ which allows us to conclude that (\ref{conq}) holds true.
\end{proof}
An immediate consequence of Lemma \ref{lfund2} is the following result.

\begin{corollary}\label{existenceCOV}
Let $\mathcal{Y} = \big((X^k)_{k\ge 1}; \mathscr{D}\big)$ be an imbedded discrete structure for a Wiener functional $X = X(0) + Y + V$ where $Y\in \mathbf{H}^2(\mathbb{F})$ and $V\in \mathbf{B}^2(\mathbb{F})$ is a continuous process. If $\mathcal{Y}$ is stable then $\langle X, B^j\rangle^\mathcal{Y} = [Y,B^j]$ exists for $j=1, \ldots, d$ and $\mathcal{E}^{2,\mathcal{Y}}(X)<\infty$.
\end{corollary}
\begin{proof}
Let $X^k = X^k(0) + M^{\mathcal{Y},k}+ N^{\mathcal{Y},k}$ be the $\mathbb{F}^k$-special semimartingale decomposition of $X^k$. If $\lim_{k\rightarrow+\infty}M^{\mathcal{Y},k}$ converges weakly in $\mathbf{B}^2(\mathbb{F})$ then it is bounded in the strong norm of $\mathbf{B}^{2}(\mathbb{F})$ and Burkholder-Davis-Gundy's inequality yields that $\sup_{k\ge 1}\mathbb{E}[M^{\mathcal{Y},k},M^{\mathcal{Y},k}](T)= \mathcal{E}^{2,\mathcal{Y}}(X) < \infty$. By applying Lemma \ref{lfund2} jointly with the predictable martingale representation theorem of the Brownian motion, we conclude $\langle X, B^j\rangle^\mathcal{Y} = [Y,B^j]$ for $j=1, \ldots, d$.
\end{proof}

We are now ready to prove Theorem \ref{towardsD1}:

\

\noindent \textbf{Proof of Theorem \ref{towardsD1}:} Throughout this proof, $C$ is a constant which may defer from line to line. In the sequel, we fix an arbitrary stable imbedded discrete structure $\mathcal{Y} = \big((X^k)_{k\ge 1}, \mathscr{D}\big)$ associated with $X\in \mathbf{B}^2(\mathbb{F})$ having a representation

$$X = X(0) + M + N$$
where $M = \sum_{j=1}^d \int H_jdB^j\in \mathbf{H}^2(\mathbb{F})$ and $N\in \mathbf{B}^2(\mathbb{F})$ is a continuous process. We claim that

$$
\mathcal{D}^\mathcal{Y}X = H.
$$

We start by observing that the stability of $\mathcal{Y}$ implies that
$$M^{\mathcal{Y},k} = \sum_{j=1}^d \oint_0 \mathbb{D}^{\mathcal{Y},k,j}XdA^{k,j}(s); k\ge 1$$
is a bounded in $\mathbf{B}^2(\mathbb{F})$ which is equivalent to $\sup_{k\ge 1}\mathbb{E}[M^{\mathcal{Y},k},M^{\mathcal{Y},k}](T)<\infty$. Moreover,

\begin{eqnarray}
\nonumber\mathbb{E}\int_0^T\|\mathbb{D}^{\mathcal{Y},k}X(s) \|^2_{\mathbb{R}^d}ds &=& \mathbb{E}[M^{\mathcal{Y},k},M^{\mathcal{Y},k}](T)\\
\label{enerA}&-&\mathbb{E}\sum_{j=1}^d\sum_{n=1}^\infty
|\mathbb{D}^{\mathcal{Y},k,j}X (T^{k,j}_n)|^2 ( T^{k,j}_{n+1} - T)1\!\!1_{\{T^{k,j}_n\le T < T^{k,j}_{n+1}\} },
\end{eqnarray}
for every $k\ge 1$. Therefore, $\sup_{k\ge 1}\mathbb{E}\int_0^T\|\mathbb{D}^{\mathcal{Y},k}X(s) \|^2_{\mathbb{R}^d}ds< \infty$.


Let us fix $g\in L^\infty$, $t\in [0,T]$ and $j=1, \ldots, d$. By the very definition, we have

\begin{eqnarray*}
\nonumber\mathbb{E}g\int_0^t\mathbb{D}^{\mathcal{Y},k,j}X(s)ds &=&\mathbb{E}g\sum_{n=1}^\infty\mathbb{D}^{\mathcal{Y},k,j}X(T^{k,j}_{n-1})\Delta T^{k,j}_n 1\!\!1_{\{T^{k,j}_{n-1}\le t\}}\\
\nonumber& &\\
\nonumber&-&\mathbb{E}g\sum_{n=1}^\infty\mathbb{D}^{k,j}X(T^{k,j}_{n-1})(T^{k,j}_n - t)1\!\!1_{\{T^{k,j}_{n-1}< t\le T^{k,j}_n\}}\\
\nonumber& &\\
&=:& I^{k,j,1}(t) + I^{k,j,2}(t).
\end{eqnarray*}
By the strong Markov property, $\Delta T^{k,j}_n$ is independent from $\mathcal{F}^{k}_{T^{k,j}_{n-1}}$ and the Brownian scaling yields $\mathbb{E} \Delta T^{k,j}_n  = \epsilon^2_k$ for every $n\ge 1$. Then, we shall estimate

$$ |I^{k,j,2}(t)| \le C \epsilon_k \sum_{n=1}^{\infty}\mathbb{E} |\Delta X^k (T^{k,j}_{n-1}) |1\!\!1_{\{T^{k,j}_{n-1}\le t < T^{k,j}_n\}}\rightarrow 0$$
as $k\rightarrow \infty$. We claim
\begin{small}
\begin{eqnarray}
\nonumber\mathbb{E}\Bigg[ g\sum_{n=1}^\infty\mathbb{D}^{\mathcal{Y},k,j}X(T^{k,j}_{n-1})\Delta T^{k,j}_n 1\!\!1_{\{T^{k,j}_{n-1}\le t\}}\Bigg] &=&\mathbb{E}\Bigg[\sum_{n=2}^\infty g^{k,j}_n \mathbb{D}^{\mathcal{Y},k,j}X(T^{k,j}_{n-1})\Delta T^{k,j}_n1\!\!1_{\{T^k_{n-1}\le t\}}\Bigg]\\
\nonumber& &\\
\label{id}&+&\mathbb{E}\Bigg[ g \sum_{n=1}^\infty \Delta X^k(T^{k,j}_{n-1})\Delta A^{k,j}(T^{k,j}_{n-1})1\!\!1_{\{T^{k,j}_{n-1}\le t\}}\Bigg],
\end{eqnarray}
\end{small}
where $g^{k,j}_n:=\mathbb{E}[g|\mathcal{F}^k_{T^{k,j}_n}] -\mathbb{E}[g|\mathcal{F}^k_{T^{k,j}_{n-1}}]; n\ge 1$.
Indeed, for each $n,k\ge 1$
\begin{small}
\begin{eqnarray*}
\int_{\{T^{k,j}_{n-1}\le t\}}g\mathbb{D}^{\mathcal{Y},k,j} X (T^{k,j}_{n-1}) \Delta T^{k,j}_n d\mathbb{P}&=& \int_{\{T^{k,j}_{n-1}\le t\}}\mathbb{E}[g|\mathcal{F}^k_{T^{k,j}_n}]\mathbb{D}^{\mathcal{Y},k,j}X (T^{k,j}_{n-1}) \Delta T^{k,j}_nd\mathbb{P}\\
& &\\
&-& \int_{\{T^{k,j}_{n-1}\le t\}}\mathbb{E}[g|\mathcal{F}^k_{T^{k,j}_{n-1}}]\mathbb{D}^{\mathcal{Y},k,j}X(T^{k,j}_{n-1}) \Delta T^{k,j}_n d\mathbb{P}\\
& &\\
&+&\int_{\{T^{k,j}_{n-1}\le t\}}\mathbb{E}[g|\mathcal{F}^k_{T^{k,j}_{n-1}}]\mathbb{D}^{\mathcal{Y},k,j}X(T^{k,j}_{n-1}) \Delta T^{k,j}_nd\mathbb{P}.
\end{eqnarray*}
\end{small}

Since $|\Delta A^{k,j}(T^{k,j}_{n-1})|^2 = \epsilon^2_k$ a.s for every $n\ge 2$, we have
\begin{small}
\begin{eqnarray*}
\int_{\{T^{k,j}_{n-1}\le t\}}\mathbb{E}[g|\mathcal{F}^{k,j}_{n-1}]\mathbb{D}^{\mathcal{Y},k,j}X(T^{k,j}_{n-1}) \Delta T^{k,j}_nd\mathbb{P}&=&\int_{\{T^{k,j}_{n-1}\le t\}}\mathbb{E}[g|\mathcal{F}^k_{T^{k,j}_{n-1}}]\mathbb{D}^{\mathcal{Y},k,j}X(T^{k,j}_{n-1}) \epsilon^2_k d\mathbb{P}\\
& &\\
&=& \int_{\{T^{k,j}_{n-1}\le t\}}\mathbb{E}[g|\mathcal{F}^k_{T^{k,j}_{n-1}}]\Delta X^k(T^{k,j}_{n-1})\Delta A^{k,j}(T^{k,j}_{n-1})d\mathbb{P}\\
& &\\
&=&\int_{\{T^{k,j}_{n-1}\le t\}}g\Delta X^k(T^{k,j}_{n-1})\Delta A^{k,j}(T^{k,j}_{n-1})d\mathbb{P}.
\end{eqnarray*}
\end{small}
Then,
\begin{small}
\begin{eqnarray*}
\int_{\{T^{k,j}_{n-1}\le t\}}g\mathbb{D}^{\mathcal{Y},k,j}X(T^{k,j}_{n-1}) \Delta T^{k,j}_n d\mathbb{P}&=&\int_{\{T^{k,j}_{n-1}\le t\}}g^{k,j}_n\mathbb{D}X^{\mathcal{Y},k,j}(T^{k,j}_{n-1}) \Delta T^{k,j}_nd\mathbb{P}\\
& &\\
&+&\int_{\{T^{k,j}_{n-1}\le t\}}g\Delta X^k(T^{k,j}_{n-1})\Delta A^{k,j}(T^{k,j}_{n-1})d\mathbb{P}.
\end{eqnarray*}
\end{small}
This shows that (\ref{id}) holds. We shall write

\begin{eqnarray*}
\nonumber I^{k,j,1}(t)&=& \mathbb{E}\sum_{n=1}^\infty g^{k,j}_n\mathbb{D}^{\mathcal{Y},k,j}X(T^{k,j}_{n-1})\Delta T^{k,j}_n1\!\!1_{\{T^{k,j}_{n-1}\le t\}}\\
\nonumber& &\\
\nonumber&+& \mathbb{E}g\sum_{n=1}^\infty\Delta X^k(T^{k,j}_{n-1})\Delta A^{k,j}(T^{k,j}_{n-1})1\!\!1_{\{T^{k,j}_{n-1}\le t\}}\\
\nonumber& &\\
&=:&I^{k,j,1}_1(t) + I^{k,j,1}_2(t).
\end{eqnarray*}
Let us write $\Delta T^{k,j}_n =(\Delta T^{k,j}_n)^{1/2}(\Delta T^{k,j}_n)^{1/2}; n\ge 1$. Then, Cauchy-Schwartz inequality, the finite energy property $\mathcal{E}^{2,\mathcal{Y}}(X)< \infty$ of $\mathcal{Y}$ and Lemma \ref{gknlemma} (see Appendix) yield

$$|I^{k,j,1}_1(t)|\le \Big(\mathbb{E}\sum_{n=1}^\infty |g^{k,j}_n|^2 \Delta T^{k,j}_n 1\!\!1_{\{T^{k,j}_{n-1}\le T\}} \Big)^{1/2}$$
$$\times \Big(\mathbb{E}\sum_{n=1}^\infty |\mathbb{D}^{\mathcal{Y},k,j}X(T^{k,j}_{n-1})|^2 \Delta T^{k,j}_{n}1\!\!1_{\{T^{k,j}_{n-1}\le T\}}\Big)^{1/2}$$
$$\le \Big(T\mathbb{E}\sup_{n\ge 1}|g^{k,j}_n|^2 1\!\!1_{\{T^{k,j}_{n-1}\le T\}}\Big)^{1/2} \Big(\mathcal{E}^{2,\mathcal{Y}}(X)\Big)^{1/2}\rightarrow 0$$
as $k\rightarrow \infty$. Therefore, we arrive at the following conclusion

\begin{equation}\label{equiderdelta}
\lim_{k\rightarrow\infty}\mathbb{E}g\int_0^t\mathbb{D}^{\mathcal{Y},k,j}X(s)ds\quad\text{exists}\Longleftrightarrow \lim_{k\rightarrow \infty}\mathbb{E}g[X^k,A^{k,j}](t)
\end{equation}
exists for $j=1,\ldots,d$ and, in this case,

\begin{equation}\label{equiderdelta0}
\lim_{k\rightarrow\infty}\mathbb{E}g\int_0^t\mathbb{D}^{\mathcal{Y},k,j}X(s)ds=\lim_{k\rightarrow \infty}\mathbb{E}g[X^k,A^{k,j}](t);~j=1,\ldots,d.
\end{equation}
By Corollary \ref{existenceCOV} and the stability of $\mathcal{Y}$, we have

\begin{eqnarray}
\nonumber\lim_{k\rightarrow\infty}\mathbb{E}g\int_0^t\mathbb{D}^{\mathcal{Y},k,j}X(s)ds&=&\lim_{k\rightarrow \infty}\mathbb{E}g[X^k,A^{k,j}](t)\\
\label{equiderdelta1}& &\\
\nonumber&=& \mathbb{E}g[M,B^j](t) = \mathbb{E}g\int_0^t H_j(s)dB^j(s),
\end{eqnarray}
for $j=1,\ldots,d$. Since $\mathcal{Y}$, $t\in [0,T]$ and $g\in L^\infty(\mathbb{P})$ are arbitrary and $\{\mathbb{D}^{\mathcal{Y},k,j}X; k\ge 1\}$ is $L^2_a(\mathbb{P}\times Leb)$-weakly relatively compact for each $j=1,\ldots, d$ and $\mathcal{Y}$, we then conclude that

$$
\mathcal{D}^{\mathcal{Y}}X = H
$$
for every stable discrete structure $\mathcal{Y}$ associated with $X$. This concludes the proof.

\

A closer look at the proof of Theorem \ref{towardsD1} yields the following result.
\begin{corollary}\label{equivDS}
Let $X= X(0) + \sum_{j=1}^d\int H_jdB^j + V$ be an $\mathbb{F}$-adapted process such that $H_j\in L^2_a(\mathbb{P}\times Leb); j=1,\ldots, d$ and $V\in \mathbf{B}^2(\mathbb{F})$ has continuous paths. Then, $\mathcal{Y}$ is a stable imbedded discrete structure for $X$ if, and only if, $\mathcal{Y}$ has finite energy and $\mathcal{D}^{\mathcal{Y}}X$ exists.
\end{corollary}
\begin{proof}
Just observe (\ref{equiderdelta}), (\ref{equiderdelta0}), (\ref{equiderdelta1}) and (\ref{enerA}).
\end{proof}
In view of Theorem \ref{towardsD1} and Corollary \ref{equivDS}, it is natural to arrive at the following definition:

\begin{definition}\label{defweakder}
Let $X= X(0) + \sum_{j=1}^d\int H_jdB^j + V$ be an $\mathbb{F}$-adapted process such that $H_j\in L^2_a(\mathbb{P}\times Leb); j=1,\ldots, d$ and $V\in \mathbf{B}^2(\mathbb{F})$ has continuous paths. We say that $X$ is \textbf{weakly differentiable} if there exists a finite energy imbedded discrete structure $\mathcal{Y} =\big((X^k)_{k\ge 1},\mathscr{D}\big)$ such that $\mathcal{D}^{\mathcal{Y}}X$ exists. The space of weakly differentiable processes will be denoted by $\mathcal{W}(\mathbb{F})$.
\end{definition}

Let $X\in \mathcal{W}(\mathbb{F})$ with a decomposition

\begin{equation}\label{derdecDEF}
X = X(0)+ \sum_{j=1}^d \int H_j dB^j + V
\end{equation}
where $H = (H_1, \ldots, H_d), H_j\in L^2_a(\mathbb{P}\times Leb); j=1,\ldots, d$ and $V\in \mathbf{B}^2(\mathbb{F})$ is a continuous process. Then, Theorem \ref{towardsD1} allows us to define

\begin{equation}\label{defDERIVATIVE}
\mathcal{D}X := \mathcal{D}^\mathcal{Y}X
\end{equation}
for every stable imbedded discrete structure $\mathcal{Y} =\big((X^k)_{k\ge 1},\mathscr{D}\big)$ w.r.t $X$ and, in this case, $\mathcal{D}X = H$. The weak differentiability notion requires existence of $\mathcal{D}^\mathcal{Y}X$ for a finite energy imbedded discrete structure $\mathcal{Y} = \big((X^k)_{k\ge 1},\mathscr{D}\big)$ and, from Theorem \ref{towardsD1}, this concept of derivative does not depend on the choice of the stable structure $\mathcal{Y}$.

\begin{remark}\label{whystability}
It is important to observe that for a given $X$ of the form (\ref{derdecDEF}), we cannot expect that $\mathcal{D}^\mathcal{Y}X=H$ holds for \textit{every} $\mathcal{Y}$ because there are imbedded discrete structures

$$X^k = X^k(0) + M^{\mathcal{Y},k} + N^{\mathcal{Y},k}$$
such that $X^k\rightarrow X$ weakly in $\mathbf{B}^2(\mathbb{F})$ but $\{M^{\mathcal{Y},k}; k\ge 1\}$ fails to converge to the martingale component of $X$. In this case, because of $[X^k,A^{k,j}] = [M^{\mathcal{Y},k},A^{k,j}]$, Lemma \ref{lfund2} and (\ref{equiderdelta}), $\mathcal{D}^\mathcal{Y}X$ may not even exists or it will not coincide with $H$. This type of phenomena is well-known in time-deterministic discretizations of filtrations. See e.g \cite{memin,coquet2} and other references therein. This is the reason why we restrict the computation of the weak derivative $\mathcal{D}X$ to stable structures and this is the best one might expect.
\end{remark}

\begin{remark}
Perhaps, the simplest class of examples which are not in $\mathcal{W}(\mathbb{F})$ is given by Brownian motion transformations with low regularity

$$f(B),\quad f\in H^{1,q}(\mathbb{R})$$
where $H^{1,q}(\mathbb{R})$ is the Sobolev space with degree of integrability $1\le q < 2$.
\end{remark}

A closer look at the proof of Theorem \ref{towardsD1} gives the following useful criteria to compute the derivative.

\begin{proposition}\label{criteriaSTABLE}
If $X\in \mathcal{W}(\mathbb{F})$ is associated with a stable imbedded discrete structure $\mathcal{Y}$, then $X$ has $\mathcal{Y}$-covariations $\langle X, B^j\rangle^\mathcal{Y}; j=1,\ldots,d$ as absolutely continuous processes and

$$\mathcal{D}_jX = \frac{d \langle X, B^j\rangle^\mathcal{Y}}{dt}; j=1,\ldots, d.$$
\end{proposition}
\begin{proof}
Just observe (\ref{equiderdelta}), (\ref{equiderdelta0}), (\ref{equiderdelta1}) and use Lemma \ref{lfund2}.
\end{proof}

\begin{proposition}\label{diffform}
If $X\in \mathcal{W}(\mathbb{F})$ is a weakly differentiable process with a decomposition

$$X(t) = X(0) + \sum_{j=1}^d\int_0^t\mathcal{D}_jX(s)dB^j(s) + V_X(t); 0\le t\le T,$$
then $V_X$ can be described by

\begin{equation}\label{allort}
V_X(\cdot)  = \lim_{k\rightarrow+\infty}\sum_{j=1}^d\int_0^\cdot\mathbb{U}^{\mathcal{Y},k,j}X(s)ds
\end{equation}
weakly in $\mathbf{B}^2(\mathbb{F})$ as $k\rightarrow +\infty$ for every stable imbedded discrete structure $\mathcal{Y}$ associated with $X$.
\end{proposition}
\begin{proof}
The proof is an immediate consequence of Proposition \ref{repdelta} and Theorem \ref{towardsD1}.
\end{proof}
It is important to stress that any $X\in \mathcal{W}(\mathbb{F})$ can be decomposed into a martingale and a non-martingale component $V_X$ which encodes all the possible orthogonal infinitesimal variations of $X$ w.r.t the Brownian motion noise and, in particular, it may have
unbounded variation paths. See section \ref{pvarsec} in this direction. Indeed, the process (\ref{allort}) encodes all the possible orthogonal infinitesimal variations of a weakly differentiable process w.r.t the Brownian motion noise. See Section \ref{ITOSECTION} for the smooth case and \cite{LOS1} for a universal variational characterization of $V_X$ as a functional of local-times.

The name weak derivative is justified by the following remark. Under regularity conditions, there exists a local description of $\mathcal{D}_jX$ which justifies the name weak derivative. For a given $t\ge 0$ and $j\in \{1, \ldots, d\}$, we set

\begin{equation}\label{localHT}
T^{t,\epsilon,j} := \inf\{s\ge 0; |B^j(t+s) - B^j(t)|=\epsilon\}
\end{equation}
where $\epsilon>0$. The stopping time $T^{t,\epsilon,j}$ localizes the $j$-th Brownian motion around the point $t$. Let $X$ be a real-valued square-integrable $\mathbb{F}$-semimartingale of the form

$$X(t) = X(0) + \sum_{j=1}^d\int_0^tH_j(s)dB^j(s) + \int_0^t Z(s)ds; 0\le t\le T.$$
Assume that $H_j$ has c\`adl\`ag paths for $j= 1, \ldots,d$. Then, a routine computation shows

\begin{equation}\label{poinder}
\lim_{\epsilon\rightarrow 0+}\frac{X(t+T^{t,\epsilon,j}) - X(t)}{B^j(t+T^{t,\epsilon,j}) - B^j(t)} = H_j(t)\quad \text{in probability}
\end{equation}
for every $t\ge 0$ and $1\le j\le d$. At first glance, one may say that regularity condition on $H = (H_1, \ldots, H_d)$ is not too strong, but the strong regularity comes from the drift term. Indeed, even in the semimartingale case, when the drift is a bounded variation process without absolutely continuous paths, the limit (\ref{poinder}) may not exist.

\subsection{Differentiable processes} \label{examplesDIFF}
The goal of this section is to show that the differentiability notion presented in Definition \ref{defweakder} covers a wide range class of Wiener functionals including non-semimartingales. It turns out that any strong Dirichlet process in the sense of Bertoin \cite{bertoin} will be weakly differentiable and hence all square-integrable continuous $\mathbb{F}$-semimartingales will be weakly differentiable as well. See also Section \ref{pvarsec} for examples concerning finite $p$-variation Wiener functionals. Let us recall the notion of a strong Dirichlet process: For a given strictly increasing sequence of $\mathbb{F}$-stopping times $\tau = (S_0, \ldots, S_n, \ldots)$ such that $S_0=0~a.s$ and $S_n\uparrow +\infty$ a.s as $n\rightarrow+\infty$, we write

$$|\tau|: = \mathbb{E}\sup_{n\ge 0}|S_{n+1} - S_n|\mathds{1}_{\{S_n\le T\}}.$$
We denote by $\mathbb{T}$ as the set of all random partitions of the above form. If $Y$ is a Wiener functional and $\tau = (S_0, \ldots, S_n, \ldots)$ is a random mesh, then  we write

$$Q_\tau^\alpha(Y):=|Y(0)|^\alpha+ \sum_{n=0}^\infty |Y(S_{n+1}) - Y(S_n)|^\alpha$$
for $1\le \alpha < +\infty$.
\begin{definition}\label{alphavardef}
A Wiener functional $Y$ is said to be of finite $\alpha$-variation if

$$\|Y\|_{\mathbf{Q}^\alpha}:=\sup_{\tau\in \mathbb{T}}\big[\mathbb{E} (Q^\alpha_\tau(Y))\big]^{1/\alpha}<\infty.$$
\end{definition}
Let us denote $\mathbf{Q}^\alpha(\mathbb{F})$ as the set of Wiener functionals with $\alpha$-finite variation. One can easily check that $(\mathbf{Q}^\alpha(\mathbb{F}),\|\cdot\|_{\mathbf{Q}^\alpha})$ is a Banach space.
Let $\mathbf{Q}^\alpha_0(\mathbb{F})$ be the linear subspace of $\mathbf{Q}^\alpha(\mathbb{F})$ constituted of the elements $Y\in \mathbf{Q}^\alpha_0(\mathbb{F})$ with $\alpha$-null variation, i.e.,

$$\lim_{|\tau|\rightarrow 0}\mathbb{E}\big(Q^\alpha_\tau(Y)\big)=0.$$
One can easily check that $\mathbf{Q}^2_0(\mathbb{F})$ is a Banach space. We refer to \cite{bertoin} for more details. Finally, the space of strong Dirichlet processes is given by

$$\mathscr{R}(\mathbb{F}):=\mathbf{H}^2(\mathbb{F})\oplus \mathbf{Q}^2_0(\mathbb{F}).$$
In other words, any $X\in \mathscr{R}(\mathbb{F})$ admits a unique decomposition
 $$X = X(0) + M + N$$
where $M\in \mathbf{H}^2(\mathbb{F})$ and $N\in \mathbf{Q}^2_0(\mathbb{F})$.
\begin{theorem}\label{strongDth1}
Let $X=X(0) + M + N$ be a strong Dirichlet process where $(M,N)$ is the canonical decomposition where $M = \sum_{j=1}^d\int H_jdB^j\in \mathbf{H}^2(\mathbb{F})$ and $N\in \mathbf{Q}^2_0(\mathbb{F})$. Then, $X\in\mathcal{W}(\mathbb{F})$ is weakly differentiable, where

\begin{equation}\label{strongDth1.1}
\mathcal{D}X = H
\end{equation}
and
\begin{equation}\label{strongDth1.2}
N = \lim_{k\rightarrow+\infty}\sum_{j=1}^d\int_0^\cdot \mathbb{U}^{\mathcal{Y},k,j}X(s)ds
\end{equation}
weakly in $\mathbf{B}^2(\mathbb{F})$ for every stable imbedded discrete structure $\mathcal{Y}$ w.r.t $X$.
\end{theorem}

\begin{remark}
Theorem \ref{strongDth1} justifies the use of the variational operators $(\mathbb{D}^{\mathcal{Y},k,j}X, \mathbb{U}^{\mathcal{Y},k,j}X)$ based on an imbedded discrete structure $\mathcal{Y}  = \big((X^k)_{k\ge 1},\mathscr{D}\big)$ associated with a large class of Wiener functionals $X$. However, we stress that in many examples the computation of $(\mathcal{D}^{\mathcal{Y},k,j}X, U^{\mathcal{Y},k,j}X)$ is good enough to extract non-trivial information from $X$ without needing to go to the limits (\ref{strongDth1.1}) and (\ref{strongDth1.2}). See the works \cite{bezerra,LEAO_OHASHI2017.1,LEAO_OHASHI2017.2}.
\end{remark}

\subsection{Differentiability and finite $p$-variation Wiener functionals}\label{pvarsec}
Before we proceed with the proof of Theorem \ref{strongDth1}, let us make a few remarks concerning finite $p$-variation processes associated with Theorem \ref{strongDth1}. Let us recall that $\mathscr{R}(\mathbb{F})$ contains a large class of Wiener functionals beyond semimartingales. For instance, it is known that (see Lemma 1.1 in \cite{bertoin1}) that
$$\mathbf{Q}^\alpha(\mathbb{F})\subset \mathbf{Q}_0^\beta(\mathbb{F})~\text{up to localization}$$
for every $1\le \alpha < \beta$. Let $\mathcal{V}_\alpha(\mathbb{F})$ be the space of all Wiener functionals of the form

$$X = X(0) + M + N,$$
where $M\in \mathbf{H}^2(\mathbb{F})$ and $N\in \mathbf{Q}^\alpha(\mathbb{F})$ for $1\le \alpha < 2$. Then, $\mathcal{V}_\alpha(\mathbb{F})\subset \mathscr{R}(\mathbb{F}); 1\le \alpha < 2$ and hence Theorem \ref{strongDth1} applies to elements of $\mathcal{V}_{\alpha}(\mathbb{F}); 1\le \alpha < 2$ up to localization. More concrete examples arise by making use of the pathwise Young integral. One typical example of a class of processes in $\mathcal{V}_\alpha(\mathbb{F})$ ($1\le \alpha < 2$) is given by

$$X(t)  = X(0) + \sum_{j=1}^d\int_0^t\mathcal{D}_j X(s) dB^j(s) + \int_0^t Y(s) dW(s)$$
where $(Y,W)$ is a pair of Wiener functionals satisfying the constraint

$$Y\in \mathbf{Q}^\beta(\mathbb{F}), W\in \mathbf{Q}^\alpha(\mathbb{F});~\frac{1}{\alpha} + \frac{1}{\beta} > 1,$$
where

$$t\mapsto \int_0^tY(s)dW(s)\in \mathbf{Q}^\alpha(\mathbb{F})$$
is interpreted in Young sense.

For a typical example, let us present the fractional Brownian motion case. The following result is known for deterministic partitions but it also holds for random partitions based on Garsia-Rodemich-Rumsey's inequality. For sake of completeness, we give the details of the proof.

\begin{lemma}\label{QnormFBM}
If $\frac{1}{2}< H < 1$, then
$$B_H\in \mathbf{Q}^p(\mathbb{F});~p > \frac{1}{H}$$
\end{lemma}
\begin{proof}
See Section \ref{AppendixQnormFBM} (Appendix) for the proof of this result.
\end{proof}
Then, $\mathcal{W}(\mathbb{F})$ contains the class of processes of the form

$$X(t)  = X(0) + \sum_{j=1}^d\int_0^t\mathcal{D}_jX(s)dB^j(s) + \int_0^t Y(s)dB_H(s),$$
where $Y\in \mathbf{Q}^\beta(\mathbb{F})$, $\frac{1}{\beta} + \frac{1}{H-\epsilon} > 1$ for $\epsilon \in (0,H)$ and $\frac{1}{2}< H <1$.

\subsection{Proof of Theorem \ref{strongDth1}}
The proof of Theorem \ref{strongDth1} will be divided into several steps. Throughout this section, we fix an arbitrary strong Dirichlet process

$$X = X(0)   + \sum_{j=1}^d \int H_jdB^j + N$$
where $N\in \mathbf{Q}_0^2(\mathbb{F})$. In view of Theorem \ref{towardsD1}, we need to show the existence of $\mathcal{D}^\mathcal{Y}X$ for a stable imbedded discrete structure $\mathcal{Y}$ associated with $X$. Since $X\in \mathscr{R}(\mathbb{F})$ is an arbitrary Dirichlet process, it is natural to guess that the use of the canonical structure $\mathcal{Y} =\big((\delta^kX)_{k\ge 1}, \mathscr{D}\big)$ will be a natural choice.

\begin{proposition}\label{Dirichletenergy}
If $X\in \mathscr{R}(\mathbb{F})$ and $\mathcal{Y} = \big((\delta^kX)_{k\ge 1}, \mathscr{D}\big)$, then $\mathcal{E}^{2,\mathcal{Y}}(X)< +\infty$. In particular, there exists a constant $C$ which does not depend on $k\ge 1$ such that

\begin{equation}\label{energyinequa}
\mathbb{E}[\delta^kX,\delta^kX](T)\le C\mathbb{E}\sum_{\ell=1}^\infty |X(T^k_\ell) - X(T^k_{\ell-1})|^21\!\!1_{\{T^{k}_\ell \le T\}}
\end{equation}
for every $k\ge 1$.
\end{proposition}
\begin{proof}
Throughout this proof, $C$ is a constant which may defer from line to line. The proof is significantly more tricky than the one-dimensional case treated in Lemma 3.1 in \cite{LEAO_OHASHI2013}. We divide the proof into three parts.

\textbf{STEP 1:} Let $\tau^{k,j}_n:=\min\{T^{k,r}_m; T^{k,r}_m > T^{k,j}_n\}$. It is easy to see that $\tau^{k,j}_n$ is an $\mathbb{F}^k$-stopping time. Indeed, for any $s\ge 0$, we have

\begin{eqnarray*}
\{\tau^{k,j}_n \le s\} &=& \{\tau^{k,j}_n\le s\}\cap \{s \le T^{k,j}_n\}\bigcup \{\tau^{k,j}_n\le s\}\cap \{T^{k,j}_n < s\}\\
& &\\
&=& \{\tau^{k,j}_n\le s\}\cap \{T^{k,j}_n < s\}\\
& &\\
&=& \bigcup_{m=1}^d\bigcup_{r\ge 1}\Big(\{T^{k,j}_n < T^{k,m}_r\}\cap \{T^{k,m}_r \le s\}\Big)
\end{eqnarray*}
where $\{T^{k,j}_n < T^{k,m}_r\}\in \mathcal{F}^{k}_{T^{k,m}_r-}$ so that $\{T^{k,j}_n < T^{k,m}_r\}\cap \{T^{k,m}_r \le s\}\in \mathcal{F}^{k}_s$ for each $j,m=1, \ldots , d$ and $n\ge 0$ and $r\ge 1$. Therefore, $\{\tau^{k,j}_n \le s\}\in \mathcal{F}^k_s$. This shows that $\tau^{k,j}_n$ is an $\mathbb{F}^k$-stopping time for each $k\ge 1, n\ge 0$ and $1\le j\le d$. Now, we notice that
$$\{\Delta \delta^kX \neq 0\}\subset \cup_{j=1}^\infty\cup_{n=0}^\infty[[\tau^{k,j}_n,\tau^{k,j}_n]]$$
and
$$\Delta \delta^kX(\tau^{k,j}_n) = \delta^kX(\tau^{k,j}_n) - \delta^kX(T^{k,j}_n)~a.s~\text{for each}~k\ge 1,n\ge 0 ~\text{and}~1\le j\le d.$$

In the sequel, we use the convention that $T^{k,j}_{-1}=T^{k,j}_0=0$. We shall write

 \begin{equation}\label{trick1}
 [\delta^kX,\delta^kX](T)  = \sum_{j=1}^d\sum_{\ell=0}^\infty |\Delta \delta^kX(\tau^{k,j}_\ell)|^2 1\!\!1_{\{\tau^{k,j}_\ell \le T\}}
 \end{equation}
where

\begin{eqnarray*}
\nonumber\Delta \delta^kX(\tau^{k,j}_\ell) &=& \mathbb{E}\big[X(\tau^{k,j}_\ell)|\mathcal{F}^k_{\tau^{k,j}_\ell}\big] - \mathbb{E}\big[X(T^{k,j}_\ell)|\mathcal{F}^k_{T^{k,j}_\ell}\big]\\
\nonumber& &\\
&=&\mathbb{E}\big[X(\tau^{k,j}_\ell)|\mathcal{F}^k_{\tau^{k,j}_\ell}\big] - \mathbb{E}\big[X(T^{k,j}_{\ell-1})|\mathcal{F}^k_{\tau^{k,j}_\ell}\big]\\
\nonumber& &\\
&+&\mathbb{E}\big[X(T^{k,j}_{\ell-1})|\mathcal{F}^k_{\tau^{k,j}_\ell}\big]- \mathbb{E}\big[X(T^{k,j}_\ell)|\mathcal{F}^k_{T^{k,j}_\ell}\big]
\end{eqnarray*}
We notice that

$$\mathcal{F}^k_{\tau^{k,j}_\ell} = \mathcal{F}^k_{T^{k,j}_\ell}\vee\mathcal{H}^{k,j}_\ell  $$
where $\mathcal{H}^{k,j}_\ell = \sigma\big(\tau^{k,j}_\ell -T^{k,j}_\ell, \Delta A^{k,m}(\tau^{k,j}_\ell); 1\le m\le d\big)$ for every $j=1, \ldots, d$ and $\ell\ge 0$. Let us denote

$$\varphi^{k,j}_\ell:=\mathbb{E}\big[X(T^{k,j}_{\ell-1})|\mathcal{F}^k_{\tau^{k,j}_\ell}\big]- \mathbb{E}\big[X(T^{k,j}_\ell)|\mathcal{F}^k_{T^{k,j}_\ell}\big]$$
By the very definition, $\varphi^{k,j}_0=0$, $\varphi^{k,j}_1 = \mathbb{E}\Big[ X(0) - X(T^{k,j}_1)|\mathcal{F}^k_{T^{k,j}_1} \Big]$. For $\ell\ge 2$, we need to work a little more.

\textbf{STEP 2:} Let us fix $\ell\ge 2$ and $1\le j\le d$. We claim for every $Z_1\in L^2(\mathcal{F}_{T^{k,j}_{\ell-1}})$ and $Z_2\in L^2(\mathcal{H}^{k,j}_\ell)$, we have

\begin{equation}\label{in}
\mathbb{E}\big[Z_1Z_2|\mathcal{F}^{k}_{T^{k,j}_\ell}\big] =\mathbb{E}\big[Z_1|\mathcal{F}^{k}_{T^{k,j}_\ell}\big]\mathbb{E}\big[Z_2|\mathcal{F}^{k}_{T^{k,j}_\ell}\big].
\end{equation}
In other words, $\mathcal{F}_{T^{k,j}_{\ell-1}}$ and $\mathcal{H}^{k,j}_\ell$ are conditionally independent given $\mathcal{F}^{k}_{T^{k,j}_\ell}$. Let us fix $\ell \ge 2$ and $j=1, \ldots, d$. Let $\mathcal{D}^{k,j}_\ell$ the sigma-algebra such that

$$\mathcal{F}^k_{T^{k,j}_\ell} = \mathcal{F}^k_{T^{k,j}_{\ell-1}}\vee \mathcal{D}^{k,j}_\ell.$$
It is easy to check that

$$\mathcal{D}^{k,j}_\ell = \sigma \Big(W^{k,m,\ell,j}\big(s\wedge (T^{k,j}_\ell-T^{k,j}_{\ell-1})\big); s\ge 0, 1\le m\le d\Big)$$
where $W^{k,m,\ell,j}(s): = A^{k,m}(s+T^{k,j}_{\ell-1}) - A^{k,m}(T^{k,j}_{\ell-1}); s\ge 0$ for $\ell \ge 2$ and $1\le j,m\le d$. We notice that

\begin{eqnarray}
\nonumber\mathbb{E}\Big[Z_1Z_2|\mathcal{F}^k_{T^{k,j}_{\ell}}\Big] &=& \mathbb{E}\Big[\mathbb{E}\big[Z_1Z_2|\mathcal{F}_{T^{k,j}_{\ell-1}}\vee \mathcal{D}^{k,j}_\ell\big]\big|\mathcal{F}^k_{T^{k,j}_\ell}\Big]\\
\label{tr3} & &\\
\nonumber&=&  \mathbb{E}\Big[Z_1\mathbb{E}\big[Z_2|\mathcal{F}_{T^{k,j}_{\ell-1}}\vee \mathcal{D}^{k,j}_\ell\big]\big|\mathcal{F}^k_{T^{k,j}_\ell}\Big]
\end{eqnarray}
Since $Z_2$ is $\mathcal{H}^{k,j}_\ell$-measurable, then there exists $f:\mathbb{R}_+\times \{-\epsilon_k, 0, \epsilon_k\}^d\rightarrow\mathbb{R}$ such that

$$
Z_2 = f\Big(\tau^{k,j}_\ell-T^{k,j}_\ell, \Delta A^{k,m}(\tau^{k,j}_\ell); 1\le m\le d\Big)~a.s
$$
More importantly, $Z_2$ is a functional of $A^{k,m}(T^{k,r}_n), T^{k,m}_n; n\ge 0, 1\le r,m\le d$. In this case,

\begin{equation}\label{MIDEN}
\mathbb{E}\big[Z_2|\mathcal{F}^k_{T^{k,j}_{\ell}}\big]=\mathbb{E}\big[Z_2|\mathcal{F}^k_{T^{k,j}_{\ell-1}}\vee \mathcal{D}^{k,j}_\ell\big] =\mathbb{E}\big[Z_2\big|\mathcal{F}_{T^{k,j}_{\ell-1}}\vee \mathcal{D}^{k,j}_\ell\big]~a.s
\end{equation}
Therefore, (\ref{MIDEN}) and (\ref{tr3}) yield (\ref{in}). The conditional independence (\ref{in}) yields

\begin{eqnarray*}
\mathbb{E}\big[X(T^{k,j}_{\ell-1})|\mathcal{F}^k_{\tau^{k,j}_\ell}\big]&=&\mathbb{E}\big[X(T^{k,j}_{\ell-1})|\mathcal{F}^k_{T^{k,j}_\ell}\vee \mathcal{H}^{k,j}_\ell\big]\\
& &\\
&=&\mathbb{E}\big[X(T^{k,j}_{\ell-1})|\mathcal{F}^k_{T^{k,j}_\ell}\big]
\end{eqnarray*}
and hence $\varphi^{k,j}_\ell =\mathbb{E}\big[X(T^{k,j}_{\ell-1}) - X(T^{k,j}_\ell)|\mathcal{F}^k_{T^{k,j}_\ell}\big]; \ell \ge 2$.

\textbf{STEP 3:} As a consequence of STEP 2, we may apply Jensen's inequality to find a positive constant $C$ (which does not depend on $k,j,n$) such that

\begin{small}
$$|\Delta \delta^kX(\tau^{k,j}_n)|^2 \le C \Bigg(\mathbb{E}\Big[|X(\tau^{k,j}_n) - X(T^{k,j}_n)|^2\big|\mathcal{F}^k_{\tau^{k,j}_n}\Big] +
\mathbb{E}\Big[|X(T^{k,j}_n) -X(T^{k,j}_{n-1})|^2 \big|\mathcal{F}^k_{\tau^{k,j}_n}\Big] $$
\begin{equation}\label{trick2}
+\mathbb{E}\Big[|X(T^{k,j}_{\ell-1}) - X(T^{k,j}_\ell)|^2\big|\mathcal{F}^k_{T^{k,j}_\ell}\Big]\Bigg)~a.s.
\end{equation}
\end{small}
Moreover, we notice that $\sum_{j=1}^d\sum_{\ell=0}^\infty\mathbb{E}\Big[|X(T^{k,j}_{\ell-1}) - X(T^{k,j}_\ell)|^2\big|\mathcal{F}^k_{T^{k,j}_\ell}\Big]1\!\!1_{\{\tau^{k,j}_\ell \le T\}}$ equals to
\begin{small}
\begin{equation}\label{trick3}
\sum_{j=1}^d\sum_{\ell=0}^\infty\mathbb{E}\Big[|X(T^{k,j}_{\ell-1}) - X(T^{k,j}_\ell)|^2\big|\mathcal{F}^k_{T^{k,j}_\ell}\Big]1\!\!1_{\{T^{k,j}_\ell \le T\}}~a.s.
\end{equation}
\end{small}
Then, (\ref{trick1}), (\ref{trick2}) and (\ref{trick3}) yield

\begin{eqnarray*}
\mathbb{E}[\delta^kX,\delta^kX](T)&\le& C \mathbb{E}\sum_{j=1}^d\sum_{\ell=0}^\infty|X(\tau^{k,j}_\ell) - X(T^{k,j}_\ell)|^21\!\!1_{\{\tau^{k,j}_\ell \le T\}}\\
& &\\
&+& C\mathbb{E}\sum_{j=1}^d\sum_{\ell=0}^\infty |X(T^{k,j}_\ell) - X(T^{k,j}_{\ell-1})|^21\!\!1_{\{\tau^{k,j}_\ell \le T\}}\\
& &\\
&+& C\mathbb{E}\sum_{j=1}^d\sum_{\ell=0}^\infty |X(T^{k,j}_{\ell-1}) - X(T^{k,j}_{\ell})|^21\!\!1_{\{T^{k,j}_\ell \le T\}}\\
& &\\
&\le& 3C  \mathbb{E}\sum_{\ell=1}^\infty |X(T^k_\ell) - X(T^k_{\ell-1})|^21\!\!1_{\{T^{k}_\ell \le T\}}\\
&\le& 3C \|X\|^2_{\mathbf{Q}^2} < \infty.
\end{eqnarray*}
This shows $\mathcal{E}^{2,\mathcal{Y}}(X)< +\infty$ and (\ref{energyinequa}). This concludes the proof.
\end{proof}

Before we proceed with the proof, let us recall some basics of semimartingale theory which will also help us to fix notation. Let $\textbf{H}^2(\mathbb{F}^k)$ be the space of all square-integrable $\mathbb{F}^k$-martingales starting at zero. From \cite{jacod}, we know that any square-integrable $\mathbb{F}^k$-martingale has bounded variation paths and it is purely discontinuous whose jumps are exhausted by $\cup_{n\ge 1} [[T^k_n,T^k_n]]$. In this case, any $Y^k\in \textbf{H}^2(\mathbb{F}^k)$ can be uniquely written as

\begin{equation}\label{smdec}
Y^k(t) = Y^{k,\mathbf{pj}}(t) - N^{k,Y^k}(t); t\ge 0,
\end{equation}
where $N^{k,Y^k}$ is an $\mathbb{F}^k$-predictable continuous bounded variation process,
$$Y^{k,\mathbf{pj}}(t):=\sum_{0< s\le t}\Delta Y^k(s); t\ge 0$$
and $Y^{k,\mathbf{pj}}(0) = N^{k,Y^k}(0)=0$. From Th. 1 and 2 in \cite{jacod}, we can always write

$$Y^{k,\mathbf{pj}}(t) = \sum_{n=1}^\infty \Delta Y^k(T^k_n) 1\!\!1_{\{T^k_n\le t\}}; t\ge 0.$$

At first, we observe that Prop 3.1 in \cite{LEAO_OHASHI2013} holds for any sequence $\{Y^k; k\ge 1\}$ of the form (\ref{smdec}).
\begin{lemma}\label{energylemma}
Let $\{Y^k; k\ge 1\}$ be a sequence of square-integrable martingales $Y^k\in \mathbf{H}^2(\mathbb{F}^k); k\ge 1$. If  $\sup_{k\ge 1}\mathbb{E}[Y^k,Y^k](T) < \infty$, then $\{Y^k; k\ge 1\}$ is $\mathbf{B}^2$-weakly relatively sequentially compact where all limit points are $\mathbb{F}$-square-integrable martingales.
\end{lemma}
\begin{proof} Let us consider $R^{k}(t): = \mathbb{E}[Y^{k}(T)|\mathcal{F}_t];~0\le t\le T$, we can apply exactly the same arguments given in the proof of Proposition 3.1 in \cite{LEAO_OHASHI2013} to show that both $\{R^k;k\ge 1\}$ and $\{Y^k;k\ge 1\}$ are $\mathbf{B}^2$-weakly relatively compact and all limit points are $\mathbb{F}$-square-integrable martingales over $[0,T]$.
\end{proof}

In the sequel, we fix $W\in \textbf{H}^2(\mathbb{F})$ starting at zero and for simplicity of notation we assume that $W$ is defined on the whole positive line $[0,+\infty)$ and write $W^k(t):=\mathbb{E}[W(T)|\mathcal{F}^k_t]; t\ge 0$. Let, $W^k(t) = W^{k,pj}(t) - N^{k,W^k}(t);t\ge 0,$ be the $\mathbb{F}^k$-special semimartingale decomposition given in (\ref{smdec}). Let $\delta ^kW = M^{k,W} +N^{k,W}$ be the special semimartingale decomposition given by (2.10) in \cite{LEAO_OHASHI2013}. Since $W\in \textbf{H}^2(\mathbb{F})$ and $\mathbb{F}^k\subset \mathbb{F}$ for every $k\ge 1$, then

$$\mathbb{E}[W(T)|\mathcal{F}^k_t] = \mathbb{E}\big[ \mathbb{E}[W(\infty)|\mathcal{F}_T]|\mathcal{F}^k_t\big]=\mathbb{E}[W(\infty)|\mathcal{F}^k_t]; 0\le t\le T$$ so that $\mathbb{E}[W(T)|\mathcal{F}^k_{T^k_n}] = \mathbb{E}[W(\infty)|\mathcal{F}^k_{T^k_n}] = \mathbb{E}[W(T^k_n)|\mathcal{F}^k_{T^k_n}]$ on $\{T^k_n \le T\}$. In other words,

\begin{equation}\label{ident1}
W^k(T^k_n)= \delta^kW(T^k_n)~\text{on}~\{T^k_n\le T\}; k\ge 1.
\end{equation}
Let us denote $Z^k:= W^k - M^{k,W}; k\ge 1$. Since $Z^k$ is a purely discontinuous martingale, then it has a decomposition of the form (\ref{smdec}).

\begin{lemma} \label{carc_W}
The sequence $\{Z^k ; k \geq 1 \}$ satisfies $\sup_{k\ge 1}\mathbb{E}[Z^k,Z^k](T) <\infty$ and

$$\Delta Z^k(T^k_n)  1\!\!1_{ \{T^k_n \leq t \} }=  \left(N^{k,W^k}(T^k_n) - N^{k,W^k}(T^k_{n-1})\right) 1\!\!1_{ \{T^k_n \leq t \} }; n\ge 1.$$
Therefore, $\lim_{k\rightarrow \infty}Z^k = 0$ weakly in $\mathbf{B}^2(\mathbb{F})$ if, and only if,

\begin{equation}\label{wkak}
[Z^k,A^{k,j}](t)=\sum_{n=1}^\infty \left(N^{k,W^k}(T^k_n) - N^{k,W^k}(T^k_{n-1})\right) \Delta A^{k,j}(T^k_n) 1\!\!1_{\{T^k_n \leq t  \} } \rightarrow 0
\end{equation}
weakly in $L^1(\mathbb{P})$ as $k\rightarrow \infty$, for every $t\in [0,T]$.
\end{lemma}
\begin{proof}
By applying Proposition \ref{Dirichletenergy} and the fact that $W\in \textbf{H}^2(\mathbb{F})$, we observe that $\sup_{k\ge 1}\mathbb{E}[\delta^kW,\delta^kW](T) < \infty$. Moreover, Burkholder-Davis-Gundy inequality yields the bound $\sup_{k\ge 1}\mathbb{E}[W^k,W^k|(T) < \infty$ and hence, $\sup_{k\ge 1} \mathbb{E} [Z^k , Z^k](T) < \infty.$ For a given $t\in [0,T]$, we have
\begin{eqnarray}
\nonumber \Delta Z^k(T^k_n) &=& \left(\Delta W^k(T^k_n) - \Delta M^{k,W}(T^k_n)\right)\\
\nonumber &=& \left( W^k(T^k_n) - W^k(T^k_n-) - \delta^k W(T^k_n) + \delta^k W(T^k_{n-1})\right) \\
&=&\label{por1} \left( W^k(T^k_n) - W^k(T^k_n-) - W^k(T^k_n) +  W^k(T^k_{n-1})\right)\\
\nonumber&=& \left(-W^k(T^k_n-) + W^k(T^k_{n-1})\right) \\
&=&\label{por2} \left(N^{k,W^k}(T^k_n) - N^{k,W^k}(T^k_{n-1})\right)
\end{eqnarray}
on $\{T^k_n \leq t \}$ for $n \ge 1$, where in (\ref{por1}) and (\ref{por2}), we have used identity (\ref{ident1}) and the fact that $N^{k,W^k}$ has continuous paths, respectively. The last statement (\ref{wkak}) is a simple application of Lemmas \ref{energylemma}, \ref{lfund2} and the predictable martingale representation of the Brownian motion.
\end{proof}

\begin{lemma}\label{MLDirich}
Let $\delta ^kW = M^{k,W} + N^{k,W}$ be the canonical $\mathbb{F}^k$-semimartingale decomposition for a Brownian martingale $W\in \mathbf{H}^2(\mathbb{F})$ starting at zero. Then,

$$
M^{k,W}\rightarrow W
$$
weakly in $\mathbf{B}^2(\mathbb{F})$ as $k\rightarrow \infty$. Moreover, $\lim_{k\rightarrow \infty}[\delta^k W, A^{k,j}](t)=[W,B^j](t)$ weakly in $L^1(\mathbb{P})$ for every $t\in [0,T], j=1,\ldots, d$.
\end{lemma}
\begin{proof}
Lemma \ref{carc_W} and the predictability of $N^{k,W^k}$ yield $\Delta Z^k(T^k_n)  1\!\!1_{\{T^k_n \le t\}}$ is $\mathcal{F}^k_{T^k_n-}$-measurable for each $n\ge 1$ and $t\ge 0$. By construction, $|\Delta A^{k,j}(T^k_{n})|>0$ only on the set $\{\aleph_1(\mathcal{A}^k_n) = j\}$. Then, the strong Markov property and the fact that $\sigma^{k,j}_n$ is independent from $\Delta T^{k,j}_n$ imply that
$$\mathbb{E}[\Delta A^{k,j}(T^k_n)|\mathcal{F}^k_{T^k_n-}]=0~a.s,~j=1,\ldots,d.$$
Therefore,

\begin{eqnarray*}
\mathbb{E}[\Delta Z^k(T^k_n)\Delta A^{k,j}(T^k_n)|\mathcal{F}^k_{T^k_n-}] &=&
\Delta Z^k(T^k_n)\mathbb{E}[\Delta A^{k,j}(T^k_n)|\mathcal{F}^k_{T^k_n-}]\\
&=& 0~a.s
\end{eqnarray*}
on $\{T^k_n \le t\}$ for each $n\ge 1$ and $t\ge 0$ and $1\le j\le d$. By applying Prop. 1.1 in \cite{lejan} on the pure jump process $[Z^k,A^{k,j}]$ given by (\ref{wkak}), we can safely state that this process is an $\mathbb{F}^k$-martingale for every $k\ge 1$. Lemma \ref{carc_W} yields

\begin{eqnarray*}
\sup_{k\ge 1} \mathbb{E} [Z^k , Z^k](T) =\sup_{k\ge 1} \mathbb{E} \sum_{n=1}^\infty  \left(N^{k,W^k}(T^k_n) - N^{k,W^k}(T^k_{n-1})\right)^2 1\!\!1_{\{T^k_n \leq T\}} < \infty,
\end{eqnarray*}
so that
$$
\mathbb{E} \Big[[ Z^k , A^{k,j}] , [Z^k , A^{k,j}]\Big](T) \le \epsilon^2_k \mathbb{E}[Z^k,Z^k](T)\le \epsilon^2_k\sup_{r\ge 1}\mathbb{E}[Z^r,Z^r](T)\rightarrow 0
$$
as $k\rightarrow \infty$. Therefore, $\lim_{k\rightarrow \infty}[Z^k , A^{k,j}] = 0$ strongly in $\mathbf{B}^2(\mathbb{F})$ so that Lemma \ref{carc_W} yields $\lim_{k\rightarrow \infty}Z^k=\lim_{k\rightarrow \infty}\big(W^k - M^{k,W}\big)=0$ weakly in $\mathbf{B}^2(\mathbb{F})$. The set $\{M^{k,W}; k\ge 1\}$ is $\mathbf{B}^2(\mathbb{F})$-weakly relatively sequentially compact where all limits points are square-integrable $\mathbb{F}$-martingales over $[0,T]$. The weak convergence $\lim_{k\rightarrow \infty}\mathbb{F}^k=\mathbb{F}$ yields $\lim_{k\rightarrow \infty}W^k=X$ strongly in $\mathbf{B}^1(\mathbb{F})$. This allows us to conclude $\lim_{k\rightarrow \infty}M^{k,W}=W$ weakly in $\mathbf{B}^2(\mathbb{F})$. As a consequence, we apply Lemma \ref{lfund2} to state that $\lim_{k\rightarrow \infty}[M^{k,W},A^{k,j}](t) = [W,B^j](t)$ weakly in $L^1(\mathbb{P})$ for each $t\in [0,T]$ and for every $j=1,\ldots, d$.
\end{proof}
At this point, we are finally able to finish the proof of Theorem \ref{strongDth1}.

\begin{corollary}\label{deltaisstable}
Under conditions of Theorem \ref{strongDth1}, the canonical imbedded discrete structure $\mathcal{Y} = \big((\delta^kX)_{k\ge 1}, \mathscr{D}\big)$ is stable for every $X\in \mathscr{R}(\mathbb{F})$ so that $\mathscr{R}(\mathbb{F})\subset \mathcal{W}(\mathbb{F})$. Moreover,

$$\langle X, B^j\rangle^{\mathcal{Y}}(t) = [X, B^j](t)~a.s; 0\le t\le T, j=1,\ldots,d.$$
\end{corollary}
\begin{proof}
Let $X = X(0) + M + N$ be the Dirichlet decomposition of $X$ where by Brownian motion predictable representation, we can select $M  = \sum_{j=1}^d M^j$ where $M^j = \int H_jdB^j;~j=1,\ldots, d$ for $H_j\in L^2_a(\mathbb{P}\times Leb)$. By definition, $\delta^kX= X(0)+ \delta^kM+ \delta^k N$ and from Lemma~\ref{deltaconvergence}, we know that $\lim_{k\rightarrow \infty}\delta^kX= X$ weakly in $\mathbf{B}^2(\mathbb{F})$. By applying Lemma \ref{MLDirich} for the martingale $M$, for each $j=1,\ldots, d$, we have

\begin{eqnarray}
\nonumber\langle X, B^j\rangle^{\mathcal{Y}}(t) &=&\lim_{k\rightarrow \infty}[\delta^kM, A^{k,j}](t) + \lim_{k\rightarrow \infty}[\delta^kN, A^{k,j}](t)\\
\nonumber& &\\
\label{enn0}&=& [M^j, B^j](t) + \lim_{k\rightarrow \infty}[\delta^kN, A^{k,j}](t);~0\le t\le T,
\end{eqnarray}
weakly in $L^1(\mathbb{P})$ as long as the second component in the right-hand side of (\ref{enn0}) converges. Let us check the $L^1$-weak convergence of $[\delta^kN, A^{k,j}](t)$. By Proposition \ref{Dirichletenergy}, we have

$$
\mathcal{E}^{2,\mathcal{Y}}(X) = \sup_{k\ge 1}\mathbb{E}\sum_{n=1}^{\infty}(\Delta\delta^kX(T^k_n))^2 1\!\!1_{\{T^k_n\le T\}}< \infty.
$$
By applying Kunita-Watanabe inequality, we get for each $t\in [0,T]$

\begin{eqnarray}
\nonumber\mathbb{E}|[\delta^kN, A^{k,j}](t)|&\le& (\mathbb{E}[\delta^kN, \delta^kN](t))^{1/2} \times (\mathbb{E}[A^{k,j},A^{k,j}](t))^{1/2}\\
\nonumber& &\\
\label{enn1}&\le& C \Big(\mathbb{E}[\delta^kN,\delta^kN](t)\Big)^{1/2}\rightarrow 0
\end{eqnarray}
as $k\rightarrow \infty$, where $C = (\max_{1\le j\le p}\sup_{k\ge 1}\mathbb{E}[A^{k,j},A^{k,j}](T))^{1/2}< \infty$. In~(\ref{enn1}), we observe $\lim_{k\rightarrow \infty}\mathbb{E}[\delta^kN,\delta^kN](t)=0$ due to (\ref{energyinequa}) and the fact that

$$\lim_{|\tau|\rightarrow 0}\mathbb{E}\big(Q^2_\tau(N)\big)=0.$$
From~(\ref{enn0}) and (\ref{enn1}), we have $\langle X, B^j\rangle^{\mathcal{Y}} = [M^j, B^j] = [X,B^j]; j=1,\ldots, d$.
\end{proof}

\subsection{It\^o processes and BSDEs}\label{ITOSECTION}
In this section, we examine the connection between It\^o processes and backward SDEs (henceforth abbreviated to BSDEs). The key point is the existence of the limit

$$\lim_{k\rightarrow+\infty}\mathbb{U}^{\mathcal{Y},k,j}X; j=1,\ldots,d.$$
This type of question is related to some regularity properties of the non-martingale component associated with a weakly differentiable process $X$ in the sense of Definition \ref{defweakder}. See also Lemma \ref{findingSTABLE}.

\begin{theorem}\label{ItoP}
If there exists an imbedded discrete structure $\mathcal{Y}$ for $X\in\mathbf{B}^2(\mathbb{F})$ such that

\begin{equation}\label{ItoPcond}
\mathcal{E}^{2,\mathcal{Y}}(X)< \infty
\end{equation}
and
\begin{equation}\label{ItoPcond1}
\{\mathbb{U}^{\mathcal{Y},k,j}X; 1\le j\le d, k\ge 1\}~\text{is uniformly integrable in}~L^1_a(Leb\times \mathbb{P}),
\end{equation}
then $X$ is an It\^o process. In particular, any square-integrable It\^o process $X$ is uniquely written in the following differential form

\begin{equation}\label{itodif}
X(t)=X(0) + \sum_{j=1}^d  \int_0^t\mathcal{D}_jX(s)dB^{j}(s) + \int_0^t\mathcal{U} X(s)ds ;~0\le t\le T,
\end{equation}
where

\begin{equation}\label{infinitesimalG}
\mathcal{U}X:=\mathcal{U}^{\mathcal{Y}}X =\lim_{k\rightarrow\infty}\sum_{j=1}^d\mathbb{U}^{\mathcal{Y},k,j}X~\text{weakly in}~L^1_a(\mathbb{P}\times Leb)
\end{equation}
for every stable discrete structure $\mathcal{Y} = \big((X^k)_{k\ge 1}, \mathscr{D}\big)$ associated with $X$ satisfying (\ref{ItoPcond}) and (\ref{ItoPcond1}).
\end{theorem}
\begin{proof}
Throughout this proof, $C$ is a generic constant which may defer from line to line. Let $X\in \mathbf{B}^2(\mathbb{F})$ be a Wiener functional admitting an imbedded discrete structure $\mathcal{Y} = \big((X^k){k\ge 1},\mathscr{D}\big)$ satisfying (\ref{ItoPcond}) and~(\ref{ItoPcond1}). Then $\mathbb{D}^{\mathcal{Y},k,j}X;~j=1,\ldots, d$ and $\sum_{i=1}^d\mathbb{U}^{\mathcal{Y},k,i}X$ are weakly relatively sequentially compact sequences in $L^2_a(\mathbb{P}\times Leb)$ and $L^1_a(\mathbb{P}\times Leb)$, respectively. Then we shall extract common weakly convergent subsequences. With a slight abuse of notation, we still denote them by $\mathbb{D}^{\mathcal{Y},k,j}X;~j=1,\ldots, d$ and $\sum_{i=1}^d\mathbb{U}^{\mathcal{Y},k,i}X$. By applying the same argument used in the proof of Theorem \ref{strongDth1} along the convergent subsequence $\{\mathbb{D}^{\mathcal{Y},k,j}X; k\ge 1\}$, there exists a vector of adapted processes $H_j\in L^2_a(\mathbb{P}\times Leb); j=1,\ldots, d$ and there exists $N\in \mathbf{B}^2(\mathbb{F})$ such that

$$X(t)  = X(0) + \sum_{j=1}^d\int_0^tH_j(s)dB^j(s) + N(t);~0\le t\le T.$$
We claim that $N(t)=\int_0^t\gamma(s)ds;~0\le t\le T$ where $\gamma:=\lim_{k\rightarrow\infty}\sum_{i=1}^d\mathbb{U}^{\mathcal{Y},k,i}X$ weakly in $L^1_a(\mathbb{P}\times Leb)$. From Proposition~\ref{repdelta}, we already know that

$$\lim_{k\rightarrow \infty}\sum_{j=1}^d\int_0^\cdot U^{\mathcal{Y},k,j}X(s) d\langle A^{k,j},A^{k,j}\rangle(s) = N(\cdot)$$
weakly in $\textbf{B}^2(\mathbb{F})$. By construction, $N$ is $\mathbb{F}$-adapted and it has continuous paths. Hence, in order to show that $N$ and $\int\gamma(s)ds$ are indistinguishable, one only has to check they are modifications from each other. It is sufficient to check for a given $g\in L^{\infty}$ ($\mathcal{F}_T$ - measurable) and $0\le t\le T$,

\begin{equation}\label{only}
\mathbb{E}g N(t) = \mathbb{E}g\int_0^t\gamma(s)ds.
\end{equation}
But this is obvious. Indeed, by the very definition and the uniqueness of the weak limit we have

\begin{eqnarray*}
\mathbb{E}g\int_0^t\sum_{j=1}^dU^{\mathcal{Y},k,j} X(s)d\langle A^{k,j},A^{k,j}\rangle(s) &=& \mathbb{E}g\int_0^t\sum_{j=1}^d\mathbb{U}^{\mathcal{Y},k,j}X(s)ds\\
&\rightarrow&\mathbb{E}g\int_0^t\gamma(s)ds= \mathbb{E}gN(t).
\end{eqnarray*}
This proves (\ref{only}). Now let us check the second part of the theorem. Let us assume that $X$ is a square-integrable It\^o process of the form

$$
X(t) = X(0) + \sum_{j=1}^d\int_0^t H_j(s)dB^j(s) + \int_0^tV(s)ds;~0\le t\le T.
$$
Let us now check that $(\mathcal{D}X, \mathcal{U}X)$ exists. As a strong Dirichlet process, from Theorem \ref{strongDth1}, we already know that $X\in\mathscr{R}(\mathbb{F})$ is weakly differentiable and $\mathcal{D} X= (H_1, \ldots, H_d)$ in $L^2_a(\mathbb{P}\times Leb)$.

\textbf{Claim}~$\mathcal{U}X = V$. Let us consider an arbitrary stable imbedded discrete structure $\mathcal{Y} = \big((X^k)_{k\ge1}, \mathscr{D}\big)$ associated with $X$. To shorten notation, let us denote $M:= \sum_{j=1}^d\int H_{j}dB^j$ and $Y := \int V(s)ds$. The $\mathbb{F}^k$-semimartingale decomposition based on $\mathcal{Y}$ is

$$X^k(t) = X^k(0) + \sum_{j=1}^d\oint_0^t \mathbb{D}^{\mathcal{Y},k,j}X(s)dA^{k,j}(s) + \sum_{j=1}^d\int_0^tU^{\mathcal{Y},k,j}X(s)d\langle A^{k,j}, A^{k,j}\rangle(s)$$
for $0\le t\le T$. Because $\mathcal{Y}$ is stable, then we shall use Proposition~\ref{repdelta} to state that

$$
\sum_{j=1}^d\int_0^\cdot U^{\mathcal{Y},k,j}X(s)d\langle A^{k,j}, A^{k,j}\rangle(s)\rightarrow  Y
$$
weakly in $\textbf{B}^2(\mathbb{F})$ as $k\rightarrow \infty$. In particular, by taking $g\in L^{\infty}(\mathcal{F}_T)$ and $t\in [0,T]$, we shall consider the bounded linear functional $S = g1\!\!1_{[0,t]}\in \text{M}^2(\mathbb{F})$ to get

\begin{eqnarray*}
\mathbb{E}g\int_0^t\sum_{j=1}^d\mathbb{U}^{\mathcal{Y},k,j}X(s)ds&=&\mathbb{E}g\int_0^t\sum_{j=1}^dU^{\mathcal{Y},k,j}X(s)d\langle A^{k,j}, A^{k,j}\rangle(s)\\
&\rightarrow& \mathbb{E}g\int_0^t V(s)ds
\end{eqnarray*}
as $k\rightarrow \infty$ for each for $g\in L^{\infty}(\mathcal{F}_T)$ and $t\in [0,T]$. Hence, (\ref{itodif}) holds true and we conclude the proof.
\end{proof}

Of course, it is already known that any It\^o process

$$X = X(0) + \sum_{j=1}^d\int H_j dB^j + \int Z ds$$
is completely characterized by $(H_1, \ldots, H_d,Z)$. The main message of Theorem \ref{ItoP} is that
$$(H_1, \ldots, H_d,Z) = (\mathcal{D}_1X, \ldots, \mathcal{D}_dX,\mathcal{U}X)$$
where $(\mathcal{D}_1X, \ldots, \mathcal{D}_dX,\mathcal{U}X)$ can be intrinsically constructed by means of any stable imbedded discrete structure $\mathcal{Y}$ satisfying (\ref{ItoPcond}) and (\ref{ItoPcond1}).


The differential operator $\mathcal{U}X$ basically describes the mean of any square-integrable It\^o process in an infinitesimal time interval

$$
\mathbb{E}X(t)  \sim \mathbb{E}X(0) + t \mathbb{E}~\mathcal{U}X(t)\quad \text{for small}~t > 0.
$$
Of course, when $X(\cdot) = g(\cdot,W(\cdot))$ is a smooth transformation of a Markovian diffusion $W$, then $\mathcal{U}X(t) = \partial_t g(t,W(t))+\mathcal{L}g(t,W(t))$, where $\mathcal{L}$ is the infinitesimal generator of $W$. This justifies the following definition.

\begin{definition}\label{wigdef}
We say that a Wiener functional $X\in\mathcal{W}(\mathbb{F})$ admits a stochastic infinitesimal generator if $\mathcal{U}X$ exists.
\end{definition}

We advocate the existence of the infinitesimal generator as a ``heat-type operator'' requires strong pathwise regularity in the sense of \cite{dupire,cont}.

\begin{remark}\label{dupirecomp}
If $F$ has pathwise $\mathbb{C}^{1,2}\big(\Lambda\big)$-regularity (in the sense of \cite{dupire}) and $X = F(Z)$ for a continuous $\mathbb{F}$-semimartingale $Z$, then
\begin{equation}\label{splitting}
\nabla^hF(Z) + \frac{1}{2}\text{tr}\nabla^{v,2}F(Z)\frac{d [Z,Z]}{dt} = \mathcal{U}X=\mathcal{U}^\mathcal{Y}X
\end{equation}
for every stable imbedded discrete structure $\mathcal{Y}$ w.r.t $X=F(Z)$, where $\nabla^h$ and $\nabla^{v,2}$ denote the horizontal and second order vertical derivatives in the sense of pathwise calculus. The existence of the splitting on the left-hand side of (\ref{splitting}) requires severe regularity either because $\mathcal{U}X$ may not exist or one of the functionals $\nabla^hF$ and $\text{tr}\nabla^{v,2}F$ may not exist. In fact, we advocate the important object is $U^{\mathcal{Y},k}X$ where $\mathcal{Y}$ ranges over all imbedded discrete structures for a given possibly non-smooth $X$. See \cite{LEAO_OHASHI2017.1,LEAO_OHASHI2017.2}) for details.
\end{remark}
Under strong regularity conditions, we shall provide a local characterization of $\mathcal{U}X$. See Proposition \ref{pointwisederProp} in the Appendix.

\subsection{Variational Representation of BSDEs} Let us now briefly illustrate the role of $(\mathcal{D}X,\mathcal{U}X)$ in the BSDEs. Let $\mathbb{H}$ be the Cameron-Martin space associated with the one-dimensional Wiener measure and let $L^2_a(\Omega; \mathbb{H})$ be the $\mathbb{P}$-equivalent class of $\mathbb{H}$-valued random variables such that

$$\mathbb{E}\|u\|^2_{\mathbb{H}}:=\mathbb{E}\int_0^T |\dot{u}(s)|^2ds< \infty$$
where the Radon-Nikodym derivative $\dot{u}$ is $\mathbb{F}$-adapted. Let us consider a well-posed BSDE (see \cite{pardoux}):

\begin{equation} \label{BSDE}
Y(t) = \xi + \int_{t}^T g (r , Y(r) , Z(r)) dr - \sum_{j=1}^d \int_{t}^T Z^j(r) d B^{j}(r), \quad 0 \leq t \leq T,
\end{equation}
where $g : \Omega \times [0,T] \times \Bbb{R} \times \Bbb{R}^d \rightarrow \mathbb{R}$ is the generator of the BSDE and $\xi \in L^2 (\mathcal{F}_T)$. A strong solution of the BSDE (\ref{BSDE}) is an $(\Bbb{R} \times \Bbb{R}^d)$-valued $\mathbb{F}$-adapted process $(Y,Z)$ which satisfies (\ref{BSDE}) almost surely. For a given $\xi\in L^2(\mathcal{F}_T)$ and under suitable technical assumptions on $g$, it is well known there exists a unique solution $(Y,Z)$.

Now, for a given pair $(g,\xi)$, let $Y$ be a square-integrable It\^o process such that $Y(T)=\xi$~a.s and we set

$$\Lambda (Y,g,\xi)(t):= Y(t) -Y(0) + \int_0^t g(s,\mathcal{D}Y(s),Y(s)\big)ds,$$

$$\mathbb{Y}^{(g,\xi)}(t):=\int_0^t\Lambda (Y,g,\xi)(s)ds;0\le t\le T.$$
In the sequel, when we write $\mathbb{Y}^{(g,\xi)}$ it is implicitly assumed that $Y$ is a square-integrable It\^o process such that $Y(T)=\xi$ where $(g,\xi)$ is given.
\begin{theorem}\label{bsdecor}
Let $\xi\in L^2(\mathcal{F}_T)$ be a fixed terminal condition. A pair $(Y,Z)$ is a strong solution of (\ref{BSDE}) if, and only if, $Y\in \mathcal{W}(\mathbb{F})$, $Z=\mathcal{D}Y$ and

\begin{equation} \label{FSDE}
\left\{\begin{array}{l}\mathcal{U} Y(t) + g (t , Y(t) , \mathcal{D} Y(t)) = 0, \quad 0 \leq t < T,~a.s \\
Y(T) = \xi~a.s.
\end{array}\right.
\end{equation}
In particular, $Y(0) = \mathbb{E}[\xi] - \mathbb{E}\int_0^T\mathcal{U}Y(s)ds$. Moreover, for a given pair $(g,\xi)$, a square-integrable It\^o process $Y$ is a solution of (\ref{FSDE}) if, and only if,

\begin{equation}\label{j6}
\mathbb{Y}^{(g,\xi)} \in \arg\min_{X\in L^2_a(\Omega;\mathbb{H}); X(T)=\mathbb{Y}^{(g,\xi)}(T)}\mathbb{E}\|X\|^2_{\mathbb{H}},
\end{equation}
or, in other words,

\begin{equation}\label{minen}
\sum_{j=1}^d \mathbb{E}\int_0^T\int_0 ^s|\mathcal{D}_ jY(r)|^2drds= \min_{X\in L^2_a(\Omega;\mathbb{H}); X(T)=\mathbb{Y}^{(g,\xi)}(T)}\mathbb{E}\int_0^T|\dot{X}(s)|^2ds.
\end{equation}
\end{theorem}
\begin{proof}
See Section \ref{Appendixbsdecor} (Appendix) for the proof of this result.
\end{proof}

\section{Appendix}
In this section, we present the proofs of Lemmas \ref{angleLemma}, \ref{jfiltration}, \ref{explicit}, \ref{QnormFBM} and
Theorem \ref{bsdecor}.
\subsection{Proof of Lemma \ref{angleLemma}}\label{AppendixangleLemma}
\begin{proof}
In the sequel, we denote $F_k$ as the distribution function of $\Delta T^{k,1}_1$ and $f_k = F'_k$. The fact that $A^{k,j}$ is an $\mathbb{F}^{k,j}$-square integrable martingale follows from e.g \cite{germano} and the fact that
$$\mathbb{E}\sup_{0\le t\le T}|A^{k,j}(t)|^p\le \mathbb{E}\sup_{0\le t\le T}|B^j(t)|^p < \infty$$
for every $p\ge 1$. By definition, the angle bracket $\langle A^{k,j},A^{k,j}\rangle $ is the $\mathbb{F}^{k,j}$-dual predictable projection of the quadratic variation $[ A^{k,j}, A^{k,j}] $. Let us define

\[
\mu_{A^{k,j}} \left([0,t] , i \right) = \sum_{n=1}^\infty 1\!\!1_{ \{ \sigma^{k,j}_n=i\}}  1\!\!1_{ \{ T^{k,j}_n \leq t \}}, \quad i \in \{-1,1\},
\]
By definition,
\[
A^{k,j}(t) = \int_{0}^t \sum_{i \in \{-1, 1 \} } \epsilon_ k  \mu_{A^{k,j}} \left(ds , i \right) ~ ~ t \geq 0  ~~ \mbox{and} ~~ j=1, \ldots, d.
\]
Moreover, by writing $\Delta T^{k,j}_{n+1} = \Delta T^{k,j}_{n+1} - \Delta T^{k,j}_{n}$, we have

\begin{eqnarray*}
\nonumber\mathbb{P} \left[ \Delta T^{k,j}_{n+1} \in [0,t] , \sigma^{k,j}_{n+1} =i \mid \mathcal{F}^{k,j}_{T^{k,j}_n} \right] &=& \int_{0}^t \mathbb{P} \left[\sigma^{k,j}_{n+1} =i  \mid \mathcal{F}^{k,j}_{T^{k,j}_n}, \Delta T^{k,j}_{n+1}=s \right] f_k(s) ds \\
&=& \int_{0}^t \mathbb{P} \left[\sigma^{k,j}_{n+1} =i \right] f_k(s) ds\\
&=& \frac{1}{2} \int_{0}^t f_k(s) ds,
\end{eqnarray*}
for $i \in \{-1,1\}, t\ge 0$ and an integer $n\geq 0$. It follows from (\cite{bremaud}, Theorem 7 pp. 238) that the $\mathbb{F}^{k,j}$-dual predictable projection of the random measure $\mu_{A^{k,j}}$ is given by

\begin{equation} \label{dpprm}
(\mu_{A^{k,j}})^{p,k} ( [0,t], i) = \frac{1}{2} \int_{0}^t \sum_{n=0}^\infty  \frac{f_k (s-T^{k,j}_n)}{1-F_k (s-T^{k,j}_n)}   1\!\!1_{ \{ T^{k,j}_n < s \leq T^{k,j}_{n+1}\}} ds,
\end{equation}
for every $0\le t\le T$. By definition, the quadratic variation is

\begin{eqnarray} \nonumber
 [A^{k,j},A^{k,j} ](t) &=& \sum_{n=1}^\infty \mid \epsilon_ k  \mid^2 \mathds{1}_{\{T^{k,j}_n \leq t \} } = \int_{0}^t \sum_{ i \in \{-1 , 1\}  } \mid \epsilon_ k i \mid^2 \mu_{A^{k,j}} \left(ds , i \right) \\ \label{rsbA}
 &=&  \int_{0}^t  \mid \epsilon_ k  \mid^2  \sum_{ i \in \{-1 , 1\}  } \mu_{A^{k,j}} \left(ds , i \right), \quad 0 \leq t \leq T.
\end{eqnarray}
Hence, by applying equations (\ref{rsbA}) and (\ref{dpprm}), we obtain that

\begin{eqnarray*}
\langle A^{k,j},A^{k,j}\rangle (t) &=& \int_{0}^t  \mid \epsilon_ k  \mid^2  \sum_{ i \in \{-1 , 1\}   } (\mu_{A^{k,j}})^{p,k} (ds, i) \\
&=& \int_{0}^t  \mid \epsilon_ k  \mid^2  \sum_{ i \in \{-1 , 1\} }  \frac{1}{2}  \sum_{n=0}^\infty  \frac{f_k (s-T^{k,j}_n)}{1-F_k (s-T^{k,j}_n)}   1\!\!1_{ \{ T^{k,j}_n < s \leq T^{k,j}_{n+1}\}} ds \\
&=&  \mid \epsilon_ k  \mid^2 \int_{0}^t  \sum_{n=0}^\infty  \frac{f_k (s-T^{k,j}_n)}{1-F_k (s-T^{k,j}_n)}   1\!\!1_{ \{ T^{k,j}_n < s \leq T^{k,j}_{n+1}\}} ds.
\end{eqnarray*}
\end{proof}

\begin{lemma}\label{l.estimativas}
Let $Z_1,\ldots,Z_n$ be an i.i.d. sequence of absolutely continuous positive random variables. Then,
for every $\alpha\in(0,1)$ and $r\ge 1$, we have
$$\mathbb{E}\left[\left(\vee_{i=1}^{n} Z_i\right)^r\right] \leq  \Big(\mathbb{E}[Z_1^{r/(1-\alpha)}]\Big)^{(1-\alpha)}n^{1-\alpha}.$$
\end{lemma}
\begin{proof}
  Let $\alpha\in(0,1)$, and let $g$ be the density of $Z_1$, with $G$ being its distribution function. Then, by taking $p=\alpha^{-1}$ and $q=(1-\alpha)^{-1}$. H\"older inequality yields
  \begin{align*}
  \mathbb{E}\left[\left(\vee_{i=1}^{n} Z_i\right)^r\right]&= n\int_0^{+\infty}t^{r}(G(t))^{n-1}g(t)\,dt\\
           &\le n\Big(\int_0^{+\infty}[t^{r}(g(t))^{1-\alpha}]^{q}\,dt\Big)^{\frac{1}{q}}\Big(\int_0^{+\infty}[(g(t))^{\alpha}(G(t))^{n-1}]^p\,dt\Big)^{\frac{1}{p}} \\
           &=n\Big(\int_0^{+\infty}t^{qr}g(t)\,dt\Big)^{\frac{1}{q}}\Big(\int_0^{+\infty}g(t)(G(t))^{(n-1)p}\,dt\Big)^{\frac{1}{p}} \\
           &=n\Big(\mathbb{E}[Z_1^{qr}]\Big)^{\frac{1}{q}}\Big(\frac{1}{p(n-1)+1}\Big)^{\alpha}\\
           &\le n^{1-\alpha}\Big(\mathbb{E}[Z_1^{qr}]\Big)^{\frac{1}{q}},
  \end{align*}
	where in the last line we used the fact that the function $n\mapsto n/(p(n-1)+1)$ is decreasing for $n\geq 1$ and $p>1$.
\end{proof}
 \subsection{Proof of Lemma \ref{jfiltration}}\label{Appendixjfiltration}
 \begin{proof}
We fix $k\ge 1$. At first, we observe that $\mathbb{P} \left[ \cap_{j=1}^d \cap_{n=1}^\infty \{T^{k,j}_n < \infty\} \right] = 1$, then we clearly have $T^k_n < \infty$ a.s for every $n\ge 1$. Moreover, the fact for each $j=1\ldots, d$, $(\Delta T^{k,j}_n)^\infty_{n=1}$ is an i.i.d sequence of strictly positive variables will all finite moments and mean equals $\epsilon_k$ yield $T^{k,j}_n\uparrow+\infty$ as $n\rightarrow+\infty$. This implies $T^k_n\uparrow+\infty$ as $n\rightarrow+\infty$. We will show that $T^k_n$ is a sequence of $\mathbb{F}^k$-stopping times by using induction over $n$. It is clear that $T^k_1$ is an $\mathbb{F}^k$-stopping time. Suppose now that $T_{n-1}^k$ is an $\mathbb{F}^k$-stopping time. For a given $t\ge 0 $, we observe that
\begin{eqnarray*}
\{T_n^k >t\} &=& \left[\{T_n^k > t\}\cap\{T_{n-1}^k>t\}\right]\cup\left[\{T_n^k > t\}\cap\{T_{n-1}^k\leq t\}\right]\\
&=&\{T_{n-1}^k>t\}\cup\left[\{T_n^k > t\}\cap\{T_{n-1}^k\leq t\}\right].
\end{eqnarray*}
The induction assumption yields $\{T_{n-1}^k>t\}\in \mathcal{F}^k_t$. We now show that $\{T_n^k > t\}\cap\{T_{n-1}^k\leq t\}\in \mathcal{F}^k_t$.
Note that
$$\{T_n^k > t\}\cap\{T_{n-1}^k\leq t\} = \bigcup_{I\in\mathcal{I}_{n-1}}\left[\bigcap_{(j,m)\in I}\{T^{k,j}_m \leq t\}\cap\bigcap_{(j,m)\in I^c}\{T^{k,j}_m>t\}\right],$$
where $\mathcal{I}_{n-1} = \{I\subset\{1,\ldots,d\}\times\mathbb{N}~\text{with}~\#I = n-1; ~\hbox{ if~}(j,m)\in I,\hbox{~then,~}\forall \tilde{m}\leq m, (j,\tilde{m})\in I\}.$ Since each $\{T^{k,j}_m \leq t\}\in \mathcal{F}^k_t$ and each $\{T^{k,j}_m>t\}\in \mathcal{F}^k_t$, it is clear that $\{T_n^k > t\}\cap\{T_{n-1}^k\leq t\}\in \mathcal{F}_t^k$, and thus, $T^k_n$ is a $\mathbb{F}^k$-stopping time.

Let us now fix $t\ge 0$. By the very definition, $A^{k,j}(s \wedge t) = A^{k,j}(s\wedge T^k_n)$ on $\{T^k_n \le t < T^k_{n+1}\}$ for every $j\in \{1, \ldots, d\}$, $s\ge 0$ and $n\ge 0$. Moreover,

$$\mathcal{F}^k_t = \bigotimes_{j=1}^d\mathcal{F}^{k,j}_t= \bigotimes_{j=1}^d\sigma(A^{k,j}(s\wedge t); s\ge 0)$$
so that

\begin{eqnarray}
\nonumber\mathcal{F}^k_t \cap \{T^k_n \le t < T^k_{n+1}\}&=& \bigotimes_{j=1}^d\sigma(A^{k,j}(s\wedge t); s\ge 0)\cap \{T^k_n \le t< T^k_{n+1}\}\\
\nonumber& &\\
\label{dtf}&=& \bigotimes_{j=1}^d\sigma(A^{k,j}(s\wedge T^k_n); s\ge 0)\cap \{T^k_n \le t< T^k_{n+1}\}
\end{eqnarray}
We set $\mathcal{G}^k_n:=\sigma(A^{k,j}(s\wedge T^k_n); s\ge 0; 1\le j \le d )$. By construction $T^k_n$ is $\mathcal{G}^k_n$-measurable for every $n\ge 1$ and

$$\mathcal{F}^k_t  = \Big\{\bigcup_{\ell=0}^\infty D_\ell \cap \{T^{k}_{\ell} \le t < T^{k}_{\ell+1}\}; D_\ell\in \mathcal{G}^{k}_{\ell}~\text{for}~\ell \ge 0 \Big\},~t\ge 0.$$
Therefore, $\mathbb{F}^k$ is a filtration of discrete type (see e.g~\cite{jacod}). By Cor. 5.57 in \cite{he}, we conclude $\mathcal{F}^k_{T^k_n} = \mathcal{G}^k_n$ for each $n\ge 0$ and (\ref{dtf}) concludes the proof.
\end{proof}
\subsection{Proof of Lemma \ref{explicit}}\label{Appendixexplicit}
\begin{proof}
For simplicity, we assume $X^k(0)=0$. At first, recall that $\mathcal{F}^k_{T^k_{n+1}-} = \mathcal{F}^k_{T^k_n}\vee \sigma(T^k_{n+1})$ (see Corollary 5.57 in \cite{he}). We claim the $\mathbb{F}^k$-dual predictable projection of $X^k$ has the representation

$$\big(X^k\big)^{p,k}(t)=\sum_{j=1}^d\int_0^tU^{\mathcal{Y},k,j}X(s)d\langle A^{k,j},A^{k,j}\rangle(s); 0\le t\le T,$$
where $\langle A^{k,j},A^{k,j}\rangle $ is the angle bracket of $A^{k,j}$. Let us fix $C\in \mathcal{P}^k$. We observe that


\begin{eqnarray*}
\mathbb{E}\int_0^T \mathds{1}_{C}(s)dX^{k}(s) &=& \sum_{j=1}^d\mathbb{E}\int_0^T\mathds{1}_C(s)\mathcal{D}^{\mathcal{Y},k,j}X(s)dA^{k,j}(s)\\
&=&\mathbb{E}\sum_{j=1}^d\int_0^T \mathds{1}_C(s)U^{\mathcal{Y},k,j}X(s)d[A^{k,j},A^{k,j}](s)\\
&=& \mathbb{E}\sum_{j=1}^d\int_0^T \mathds{1}_C(s)U^{\mathcal{Y},k,j}X(s)d\langle A^{k,j},A^{k,j}\rangle(s)\\
&=&\mathbb{E}\int_0^T\mathds{1}_{C}(s)d(X^k)^{p,k}(s).
\end{eqnarray*}
This proves the first claim. Now, let us denote

$$Q^k_j=\cup_{\ell=1}^\infty [[T^{k,j}_\ell,T^{k,j}_\ell]]; j=1\ldots, d.$$
It is important to observe that the support of $\mu_{[A^{k,j}]}$ is $\text{supp}~(\mu_{[A^{k,j}]}) = Q^k_j$ for every $k\ge 1$ and $j=1,\ldots, d$. Let us fix an integer $n\ge 0$. For a given $E\in \mathcal{F}^k_{T^k_{n+1}-}$, we can choose (see Th 31 in \cite{bremaud}, page 337) an $\mathbb{F}^k$-predictable process $H$ such that

$$H(T^k_{n+1}) = \mathds{1}_E~a.s$$
and it is null outside the stochastic interval $]]T^k_n,T^k_{n+1}]] = \{(t,\omega); T^k_n(\omega)< t\le T^k_{n+1}(\omega)\}$. Then, it follows from the first part that
$$
\mathbb{E}\Big[\mathds{1}_E \Delta X^k(T^k_{n+1})\mathds{1}_{\{T^k_{n+1}\le T\}}\Big] =
\mathbb{E}\int_0^T H(s)dX^k(s)
$$
$$=\sum_{j=1}^d\mathbb{E}\int_0^TH(s)U^{\mathcal{Y},k,j}X(s)d[A^{k,j},A^{k,j}]
= \sum_{j=1}^d\int_{]]T^k_n,T^k_{n+1}]]}HU^{\mathcal{Y},k,j}Xd\mu_{[A^{k,j}]}
$$
so that 
\begin{eqnarray*}
\mathbb{E}\Big[\mathds{1}_E \Delta X^k(T^k_{n+1})\mathds{1}_{\{T^k_{n+1}\le T\}}\Big]&=&\sum_{j=1}^d\int_{[[T^k_{n+1},T^k_{n+1}]]\cap Q^k_j}HU^{\mathcal{Y},k,j}Xd\mu_{[A^{k,j}]}\\
&+& \sum_{j=1}^d\int_{[[T^k_{n+1},T^k_{n+1}]]\cap (Q^k_j)^c}HU^{\mathcal{Y},k,j}Xd\mu_{[A^{k,j}]}\\
&=&\sum_{j=1}^d\int_{[[T^k_{n+1},T^k_{n+1}]]\cap Q^k_j}HU^{\mathcal{Y},k,j}Xd\mu_{[A^{k,j}]}\\
&=&\epsilon^2_k\sum_{j=1}^d\int_{E\cap\{T^k_{n+1}\le T, \aleph_1(\eta^k_{n+1})=j\}} U^{\mathcal{Y},k,j}X(T^k_{n+1})d\mathbb{P},
\end{eqnarray*}
where $\aleph_1$ is given by (\ref{alephamap}). At this point, it is important to observe that the support of $U^{\mathcal{Y},k,j}X$ is $Q^k_j$ so that one can choose a version of the conditional expectation $\mathbb{E}_{\mu_{[A^{k,j}]}}\big[\mathcal{D}^{\mathcal{Y},k,j}X/\Delta A^{k,j}|\mathcal{P}^k\big]$ such that

$$
U^{\mathcal{Y},k,j}X(T^k_{n+1})= U^{\mathcal{Y},k,j}X(T^k_{n+1})\mathds{1}_{\{\aleph_1(\eta^k_{n+1})=j\}}~a.s
$$
for each $1\le j\le d$. Since $E\in \mathcal{F}^k_{T^k_{n+1}-}$ is arbitrary, the above computation shows (\ref{ger1}) holds true.
\end{proof}

\begin{lemma}\label{gknlemma}
If $g\in L^\infty(\mathbb{P})$, then

$$\sup_{n\ge 1}|\mathbb{E}[g|\mathcal{F}^k_{T^{k}_n}] - \mathbb{E}[g|\mathcal{F}^k_{T^{k}_{n-1}}]|1\!\!1_{\{T^{k}_n \le T\}}\rightarrow 0$$
in $L^p(\mathbb{P})$ as $k\rightarrow\infty$ for every $p > 1$.
\end{lemma}
\begin{proof}
Let $Z(t) = \mathbb{E}[g|\mathcal{F}_t]; 0\le t\le T$. Since $\mathbb{F}^k\rightarrow \mathbb{F}$ weakly as $k\rightarrow +\infty$, then the proof is straightforward. For sake of completeness, we give the details here. We write

\begin{eqnarray*}
\big|\mathbb{E}[g|\mathcal{F}^k_{T^k_n}] - \mathbb{E}[g|\mathcal{F}^k_{T^k_{n-1}}]\big|\mathds{1}_{\{T^k_n\le T\}}&\le & \big|\mathbb{E}[g|\mathcal{F}^k_{T^k_n}] - \mathbb{E}[g|\mathcal{F}_{T^k_n}]\big|\mathds{1}_{\{T^k_n\le T\}}\\
& &\\
&=& \big|\mathbb{E}[g|\mathcal{F}_{T^k_n}] - \mathbb{E}[g|\mathcal{F}_{T^k_{n-1}}]\big|\mathds{1}_{\{T^k_n\le T\}}\\
& &\\
&+& \big|\mathbb{E}[g|\mathcal{F}_{T^k_{n-1}}] - \mathbb{E}[g|\mathcal{F}^k_{T^k_{n-1}}]\big|\mathds{1}_{\{T^k_n\le T\}}\\
& &\\
&\le& 2\sup_{0\le t\le T}\big|\mathbb{E}[g|\mathcal{F}^k_t] - \mathbb{E}[g|\mathcal{F}_t]\big|\\
 & &\\
&+& \sup_{k\ge 1}\big|\mathbb{E}[g|\mathcal{F}_{T^k_n}] - \mathbb{E}[g|\mathcal{F}_{T^k_{n-1}}]\big|\mathds{1}_{\{T^k_n\le T\}}\\
& &\\
&=:& I^k_1 + I^k_2.
\end{eqnarray*}
The weak convergence $\mathbb{F}^k\rightarrow \mathbb{F}$ as $k\rightarrow +\infty$ and the fact that $g$ is bounded yield $I^k_1\rightarrow 0$ in $L^p(\mathbb{P})$ for every $p\ge 1$ as $k\rightarrow+\infty$. By using Lemma \ref{meshlemma}, the continuity of the martingale $\mathbb{E}[g|\mathcal{F}_\cdot]$ and the boundedness of $g$, one can easily check that $I^k_2\rightarrow 0$ in $L^p(\mathbb{P})$ for every $p\ge 1$ as $k\rightarrow+\infty$.
\end{proof}

\subsection{Proof of Lemma \ref{QnormFBM}}\label{AppendixQnormFBM}
\begin{proof}
Take $pH >1, 1-H < \beta < 1/2$, $0< \epsilon < \beta-1+ H$ and $p(H-\epsilon)> 1$. Let us take an arbitrary random partition of stopping times $\tau = (S_n)_{n\ge 1}$. It is well-known (see e.g the proof of Lemma 1.17.1 in \cite{mishura}) there exists $G_{T,\epsilon,\beta}\in \cap_{q\ge 1}L^q(\mathbb{P})$ and a deterministic constant $C$ such that

$$|B_H(t) - B_H(s)|\le C |t-s|^{(H-\epsilon)}G_{T,\epsilon,\beta}~a.s$$
for every $t,s\in [0,T]$. This allows us to estimate
$$|B_H(S_n) - B_H(S_{n-1})|^p\le C^p |S_n - S_{n-1}|^{p(H-\epsilon)}G_{T,\epsilon,\beta}^p~a.s$$
for every $n\ge 1$. More importantly,

\begin{eqnarray}
\nonumber\sum_{n=1}^\infty |B_H(S_n) - B_H(S_{n-1})|^p\mathds{1}_{\{S_n\le T\}}&\le& C^p \sum_{n=1}^\infty| S_n- S_{n-1}|^{p(H-\epsilon)}\mathds{1}_{\{S_n\le T\}}G_{T,\epsilon,\beta}^p\\
\nonumber& &\\
\nonumber&\le& C^p G^p_{T,\epsilon,\beta}\sup_{t_i\in \Pi}\sum_{t_i\in \Pi}|t_n - t_{n-1}|^{p(H-\epsilon)}\\
\nonumber& &\\
\label{ph1}&\le& C^p G^p_{T,\epsilon,\beta}T^{p(H-\epsilon)}~a.s
\end{eqnarray}
for every $n\ge 1$, where sup above is taken over all deterministic partitions $\Pi$ of $[0,T]$. Inequality (\ref{ph1}) is due to the fact that $p(H-\epsilon)> 1$. Then, $\|B_H\|^p_{\mathbf{Q}^p}\le C^pT^{p(H-\epsilon)} \mathbb{E}G^p_{T,\epsilon,\beta} < \infty$.
\end{proof}
\subsection{Local description of the weak infinitesimal generator}\label{AppendixlocalDESCRIPTION}
The name weak can also be justified by the following result. Not surprisingly, under strong regularity conditions, there exists a local description of the infinitesimal generator. In the sequel, $T^{t_0,\epsilon,j}$ is the hitting time given by (\ref{localHT}).

\begin{proposition}\label{pointwisederProp}
Let $X$ be a real-valued square-integrable $\mathbb{F}$-semimartingale of the form

$$X(t) = X(0) + \sum_{j=1}^d\int_0^t\mathcal{D}_jX(s)dB^j(s) + \int_0^t \mathcal{U}X(s)ds; t\ge 0.$$
where $\mathcal{D}_jX$ has absolutely continuous paths for $1\le j\le d$. Assume that both the weak derivative $\nabla^w \mathcal{D}_jX$ and $\mathcal{U}X$ are a.s continuous at $t_0$ for $1\le j\le d$. Then, for any $j\in \{1, \ldots, d\}$

$$
\lim_{\epsilon\rightarrow 0^+}\mathbb{E}\Bigg[\frac{X(t_0 + T^{t_0,\epsilon,j}) - X(t_0)}{T^{t_0,\epsilon,j}}\big|\mathcal{F}_{t_0}\vee \sigma (T^{t_0,\epsilon,j})\Bigg] = \mathcal{U}X(t_0)\quad \text{in probability}.
$$
\end{proposition}
\begin{proof}
We fix $j\in \{1,\ldots, d \}$. To shorten notation, we denote $\Delta Y(t_0 + T^{t_0,\epsilon,j}):= Y(t_0 + T^{t_0,\epsilon,j}) - Y(t_0)$ for an $\mathbb{F}$-adapted process $Y$. At first, we notice that
$\lim_{\epsilon\rightarrow 0^+}T^{t_0,\epsilon,j}=0$ in $L^1(\mathbb{P})$. We shall write

\begin{eqnarray*}
\frac{\Delta X(t_0+ T^{t_0,\epsilon,j})}{T^{t_0,\epsilon,j}} &=& \frac{1}{T^{t_0,\epsilon,j}}\int_{t_0}^{t_0+ T^{t_0,\epsilon,j}}\mathcal{U}X(s)ds\\
&+& \sum_{i=1}^d\frac{1}{T^{t_0,\epsilon,j}}\int_{t_0}^{t_0+ T^{t_0,\epsilon,j}}\mathcal{D}_iX(s)B^i(s).
\end{eqnarray*}
Since $\mathcal{U}X$ is continuous at $t_0$ a.s, then $\frac{1}{T^{t_0,\epsilon,j}}\int_{t_0}^{t_0+ T^{t_0,\epsilon,j}}\mathcal{U}X(s)ds\rightarrow \mathcal{U}X(t_0)$ in probability as $\epsilon\rightarrow 0^+$. One can check $\mathcal{F}_{t_0}\vee \sigma(T^{t_0,\epsilon,j})\rightarrow \mathcal{F}_{t_0}$ weakly (in the sense of \cite{coquet1}) as $\epsilon\downarrow 0$, then we shall use Doob inequality for martingales to state that
\begin{equation}\label{st5}
\lim_{\epsilon\rightarrow 0^+}\mathbb{E}\Bigg[\frac{1}{T^{t_0,\epsilon,j}}\int_{t_0}^{t_0 + T^{t_0,\epsilon},j}\mathcal{U}X(s)ds|\mathcal{F}_{t_0}\vee \sigma(T^{t_0,\epsilon,j})\Bigg] = \mathcal{U}X(t_0)
\end{equation}
in probability. We now claim that for each $i\in \{1, \ldots, d\}$, we have

\begin{equation}\label{st6}
\lim_{\epsilon\rightarrow 0^+}\mathbb{E}\Bigg[\frac{1}{T^{t_0,\epsilon,j}}\int_{t_0}^{t_0 + T^{t_0,\epsilon,j}}\mathcal{D}_iX(s)dB^i(s)\big|\mathcal{F}_{t_0}\vee\sigma(T^{t_0,\epsilon,j})\Bigg] = 0\quad \text{in probability}.
\end{equation}

We fix $i\in \{1, \ldots, d\}$. By assumption $\mathcal{D}_iX$ has absolutely continuous paths and $\nabla^w \mathcal{D}_iX (\cdot)$ is continuous at $t_0$ a.s. In this case, integration by parts yields

\begin{eqnarray*}
\int_{t_0}^{t_0+T^{t_0,\epsilon,j}}\mathcal{D}_iX(s)dB^i(s) &=& B^i(t_0 + T^{t_0,\epsilon,j})\mathcal{D}_iX(t_0+T^{t_0,\epsilon,j})\\
&-& B^i(t_0)\mathcal{D}_iX(t_0)- \int_{t_0}^{t_0+T^{t_0,\epsilon,j}}B^i(s)\nabla^w \mathcal{D}_iX(s)ds.
\end{eqnarray*}
Similarly to (\ref{st5}), we have

\begin{equation}\label{st7}
\lim_{\epsilon\rightarrow 0^+}\mathbb{E}\Bigg[\frac{1}{T^{t_0,\epsilon,j}}\int_{t_0}^{t_0 + T^{t_0,\epsilon,j}}B^i(s)\nabla^w \mathcal{D}_iX(s)ds\big|\mathcal{F}_{t_0}\vee\sigma(T^{t_0,\epsilon,j})\Bigg] = B^i(t_0)\nabla^w\mathcal{D}X(t_0),
\end{equation}
in probability. Let us write

\begin{eqnarray*}
 B^i(t_0 + T^{t_0,\epsilon,j})\mathcal{D}_iX(t_0+T^{t_0,\epsilon,j})&=& B^i(t_0)\mathcal{D}_iX(t_0)\\
&+& B^i(t_0+T^{t_0,\epsilon,j})\int_{t_0}^{t_0+T^{t_0,\epsilon,j}}\nabla^w\mathcal{D}_iX(s)ds\\
&+& \mathcal{D}_iX(t_0)\Delta B^i(t_0+T^{t_0,\epsilon,j}).
\end{eqnarray*}
Similarly to (\ref{st5}) and by using the regularity assumption of $\nabla^w \mathcal{D}_iX$ at $t_0$, we  know that
\begin{small}
$$\lim_{\epsilon\rightarrow 0^+}\mathbb{E}\Big[B^i(t_0+T^{t_0,\epsilon,j})\frac{1}{T^{t_0,\epsilon,j}}\int_{t_0}^{t_0+T^{t_0,\epsilon,j}}\nabla^w\mathcal{D}_iX(s)ds|\mathcal{F}_{t_0}\vee \sigma(T^{t_0,\epsilon,j})\Big] = B^i(t_0)\nabla^w \mathcal{D}_iX(t_0)$$
\end{small}
in probability. The strong Markov property of the Brownian motion yields
\begin{small}
\begin{eqnarray*}
\mathbb{E}\Big[\frac{\mathcal{D}_iX(t_0)}{T^{t_0,\epsilon,j}}\Delta B^i(t_0+T^{t_0,\epsilon,j})|\mathcal{F}_{t_0}\vee \sigma(T^{t_0,\epsilon,j})\Big] &=& \frac{\mathcal{D}_iX(t_0)}{T^{t_0,\epsilon,j}}\\
&\times &\mathbb{E}\Big[\Delta B^i(t_0+T^{t_0,\epsilon,j})|\mathcal{F}_{t_0}\vee \sigma(T^{t_0,\epsilon,j})\Big]\\
&=&0~a.s.
\end{eqnarray*}
\end{small}
From (\ref{st7}), we then have shown that (\ref{st6}) holds true. This concludes the proof.
\end{proof}
\subsection{Proof of Theorem \ref{bsdecor}}\label{Appendixbsdecor}

Let $\mathbb{H}$ be the Cameron-Martin space associated with the one-dimensional Wiener measure and let $L^2_a(\Omega; \mathbb{H})$ be the $\mathbb{P}$-equivalent class of $\mathbb{H}$-valued random variables such that

$$\mathbb{E}\|u\|^2_{\mathbb{H}}:=\mathbb{E}\int_0^T |\dot{u}(s)|^2ds< \infty$$
where the Radon-Nikodym derivative $\dot{u}$ is $\mathbb{F}$-adapted. One can readily see that the linear operator

$$m:L^2(\mathcal{F}_T)\rightarrow L^2_a(\Omega; \mathbb{H}): F\mapsto m(F):=\int_0^\cdot\mathbb{E}[F|\mathcal{F}_s]ds$$
is bounded and its adjoint $m^*$ is given by $m^*(u):=\int_0^T\dot{u}(s)ds$. Therefore, we have the following orthogonal decomposition

$$
L^2_a(\Omega;\mathbb{H}) = (ker~m^*) \bigoplus (ker~m^*)^\perp.
$$
In particular, a classical result in functional analysis says that $(ker~m^*)^\perp  = \overline{m(L^2(\mathcal{F}_T))}$ in $L^2_a(\Omega; \mathbb{H})$. By construction, we then observe that

$$(ker~m^*)^\perp = \{Y\in L^2_a(\Omega; \mathbb{H}); \dot{Y}~\text{is an}~\mathbb{F}-\text{martingale}\}.$$
Therefore, for each $X\in L^2_a(\Omega; \mathbb{H})$, there exists a unique pair $(M^X,N^X)$ which realizes

\begin{equation}\label{direct}
X(\cdot) = \int_0^\cdot M^X(s)ds  + \int_0^\cdot N^X(s)ds
\end{equation}
where $M^X$ is an $\mathbb{F}$-martingale and $\int_0^T N^X(s)ds=0~a.s$. In particular,

$$\mathbb{E}\|X\|^2_{\mathbb{H}} = \mathbb{E}\int_0^T |M^X(s)|^2ds +  \mathbb{E}\int_0^T |N^X(s)|^2ds; X\in L^2_a(\Omega;\mathbb{H}).$$
In the sequel, we denote $\mathcal{M} :=(ker~m^*)^\perp$ and $\text{proj}_\mathcal{M}:L^2_a(\Omega; \mathbb{H})\rightarrow\mathcal{M}$ is the orthogonal projection operator.
\begin{proposition}\label{basicres}
Let us fix $F\in L^2(\mathcal{F}_T)$ and we assume the set $\{X\in L_a^2(\Omega; \mathbb{H}); X(T)=F\}$ is not empty. If $X\in \mathcal{M}$ and $X(T) = F$, then

\begin{equation}\label{j1}
X \in \arg\min_{Y\in L^2_a(\Omega;\mathbb{H}); Y(T)=F}\mathbb{E}\|Y\|^2_{\mathbb{H}}.
\end{equation}
Reciprocally, if $X$ satisfies (\ref{j1}), then $X\in \mathcal{M}$.
\end{proposition}
\begin{proof}
Let us denote $\mathcal{G}(F) = \{Y\in L^2_a(\Omega;\mathbb{H}); Y(T)=F\}$. Assume that $X\in \mathcal{G}(F)\cap \mathcal{M}$. By the very definition, $m^*(Y) = m^*(X)=F$ for every $Y\in \mathcal{G}(F)$. In other words, $Y-X\in \text{ker}~m^*$ for every $Y\in \mathcal{G}(F)$. By writing $Y = Y-X + X$, we then have

$$\mathbb{E}\|Y\|^2_{\mathbb{H}} = \mathbb{E}\|Y - X\|^2_{\mathbb{H}} + \mathbb{E}\|X\|^2_{\mathbb{H}}\ge \mathbb{E}\|X\|^2_{\mathbb{H}}~\forall Y\in \mathcal{G}(F).$$
Reciprocally, let us assume that $X$ satisfies (\ref{j1}). At first, we observe that if $u\in \mathcal{G}(F)$, then by using the direct sum decomposition (\ref{direct}), we have

\begin{eqnarray*}
F&=&u(T) = \int_0^T M^u(s)ds + \int_0^T N^u(s)ds\\
& &\\
&=&\int_0^T M^u(s)ds
\end{eqnarray*}
because $\int_0^T N^u(s)ds=0~a.s$. This shows that $\text{proj}_\mathcal{M}(u)\in \mathcal{G}(F)$ for every $u\in \mathcal{G}(F)$. More importantly, by repeating the above argument, we actually have

\begin{equation}\label{j2}
\text{proj}_\mathcal{M}(u) \in \arg\min_{Y\in \mathcal{G}(F)}\mathbb{E}\|Y\|^2_{\mathbb{H}}
\end{equation}
for every $u\in \mathcal{G}(F)$. We now claim that

\begin{equation}\label{claim1}
X\in \arg\min_{Y\in \mathcal{G}(F)}\mathbb{E}\|Y\|^2_{\mathbb{H}}\Longrightarrow X \in \mathcal{M}.
\end{equation}
Indeed, if $X$ realizes (\ref{claim1}), then by using (\ref{j2}), we have

$$
\|X\|_{L^2_a(\Omega;\mathbb{H})} = \|\text{proj}_\mathcal{M}(Y)\|_{L^2_a(\Omega;\mathbb{H})}
$$
for every $Y\in\mathcal{G}(F)$ and this implies

\begin{equation}\label{j4}
\|X\|_{L^2_a(\Omega;\mathbb{H})} = \|\text{proj}_\mathcal{M}(X)\|_{L^2_a(\Omega;\mathbb{H})}.
\end{equation}
Let us write $X = (X-\text{proj}_\mathcal{M}(X))+\text{proj}_\mathcal{M}(X)$ where $(X-\text{proj}_\mathcal{M}(X))\in ker~(\text{proj}_{\mathcal{M}})$. By recalling that $\text{Range}~^\perp~(\text{proj}_\mathcal{M}) = ker~(\text{proj}_{\mathcal{M}})$, we then have

\begin{equation}\label{j5}
\|X\|^2_{L^2_a(\Omega;\mathbb{H})} = \|X-\text{proj}_\mathcal{M}(X)\|^2_{L^2_a(\Omega;\mathbb{H})} + \|\text{proj}_\mathcal{M}(X)\|^2_{L^2_a(\Omega;\mathbb{H})}.
\end{equation}
From (\ref{j4}) and (\ref{j5}), we must have $\|X-\text{proj}_\mathcal{M}(X)\|_{L^2_a(\Omega;\mathbb{H})}=0 = \text{dist}(X,\mathcal{M})$. Since $\mathcal{M}$ is closed, we then conclude that $X\in \mathcal{M}$.
\end{proof}
We are now able to prove Theorem \ref{bsdecor}.

\begin{proof}
Assume that $(Y,Z)$ is a strong solution of (\ref{BSDE}). Then, $Y$ is a square-integrable It\^o process of the form

$$Y(t) = Y(0) + \sum_{j=1}^d \int_0^tZ^j(s)dB^j(s) - \int_0^t g(s,Y(s),Z(s))ds; 0\le t \le T.$$
By Theorem \ref{ItoP}, we must have $\mathcal{D}_jY = Z^j; 1\le j\le d$ and the following identity holds true

$$\mathcal{U}Y(t) + g(t,Y(t),\mathcal{D}Y(t))=0; 0\le t\le T.$$
This shows (\ref{FSDE}). Reciprocally, if $Y$ is a square-integrable It\^o processes satisfying (\ref{FSDE}), then applying again Theorem \ref{ItoP}, we must have

$$Y(t) = Y(0) + \sum_{j=1}^d\int_0^t\mathcal{D}_j Y(s)dB^j(s) -\int_0^tg(s,Y(s),\mathcal{D}Y(s))ds; 0\le t\le T$$
where $Y(T)=\xi$ a.s. This shows that $Y$ solves (\ref{BSDE}). In particular, the initial condition is

$$Y(0) = \mathbb{E}\Big[\xi + \int_0^Tg(s,Y(s),\mathcal{D}Y(s)ds \Big]=\mathbb{E}\Big[\xi - \int_0^T\mathcal{U}Y(s)ds \Big].$$
Now, let $G(g,\xi) := \{X\in L^2_a(\Omega;\mathbb{H}); X(T)=\mathbb{Y}^{(g,\xi)}(T)\}$. Of course, $G(g,\xi)$ is  not empty. Let $Y$ be a solution of (\ref{FSDE}). Then, $\mathcal{U}Y+g(Y,\mathcal{D}Y)=0$ and hence

$$Y(\cdot) - Y(0) + \int_0^\cdot g(s,Y(s),\mathcal{D}Y(s))ds~\text{is a Brownian martingale}.$$
Therefore, we shall apply Proposition \ref{basicres} to conclude (\ref{j6}). Reciprocally, suppose that (\ref{j6}) holds true. Then, we apply Proposition \ref{basicres} to state that $\Lambda (Y,g,\xi)$ is a martingale. By applying Theorem \ref{ItoP}, we conclude that $Y$ is a solution of (\ref{FSDE}).
\end{proof}

\section*{Acknowledgements}
The authors are grateful to several referees and to an Associate Editor for their
careful reading of the first versions of the manuscript which has allowed them to
considerably improve the quality of the paper. A.Ohashi would like to thank UMA-ENSTA-ParisTech for the very kind hospitality during the last stage of this project as well as Francesco Russo for stimulating discussions and suggestions that helped shaping the theory. He also acknowledges the financial support from ENSTA ParisTech.

\end{document}